\documentclass[9pt]{article}
\usepackage{latexsym}
\usepackage{amssymb}
\usepackage{amsmath}
\usepackage{color}
\usepackage{graphicx}
\usepackage{amsthm}
\allowdisplaybreaks

\newcommand \bel {\begin{equation}\label}
\newcommand \ee {\end{equation}}
\newcommand \be {\begin{equation}}

\newcommand \RR {\mathbb R}

\newcommand \CC {\mathbb C}
\newcommand \LL {\mathbb L}
\newcommand \NN {\mathbb N}

\newcommand \del \partial

\newcommand \Bcal {\mathcal B}

\newcommand \Rcal {\mathcal R}

\newcommand \Lcal {\mathcal L}

\newcommand \Fcal {\mathcal F}

\newcommand \bei {\begin{itemize}}
\newcommand \eei {\end{itemize}}

\def\pdt{\partial_t}

\def \eps {\varepsilon}

\newtheorem{theorem}{\color{black}\indent Theorem}[section]
\newtheorem{lemma}{\color{black}\indent Lemma}[section]
\newtheorem{proposition}{\color{black}\indent Proposition}[section]

\newtheorem{remark}{\color{black}\indent Remark}[section]

\usepackage{titletoc,tocloft}
\dottedcontents{section}[1.5em]{}{1.3em}{.6em}

\begin{document}
\large
\title{\bf Global stability dynamics of the timelike extremal hypersurfaces in Minkowski space}
\author{Weiping Yan \footnote{
E-mail:\, {\tt yanwp@gxu.edu.cn (W. Yan)},
{\tt 2107401015@st.gxu.edu.cn} (W.J. Li). Corresponding author: Weiping Yan.},~~Weijia Li\\
\footnotesize College of Mathematics and Information Science, Guangxi University, Nanning 530004, P.R. China.\\

}

\date{17 December, 2021}

\maketitle
\begin{abstract}
This paper aims to study the relationship between the timelike extremal hypersurfaces and the classical minimal surfaces. This target also gives the long time dynamics of timelike extremal hypersurfaces in Minkowski spacetime $\mathbb{R}^{1+M}$ with the dimension $2\leq M\leq7$.
In this dimension, the stationary solution of timelike extremal hypersurface equation is the solution of classical minimal surface equation
$$
{\rm div}\left(\frac{\nabla u}{\sqrt{1+|\nabla u|^2}}\right)=0, \quad \forall x\in\RR^M,
$$
 which only admits the hyperplane solution by Bernstein theorem.
We prove that this hyperplane solution as the stationary solution of timelike extremal hypersurface equation is asymptotic stablely by finding the hidden dissipative structure of linearized equation.
Here we  overcome that the vector field method (based on the energy estimate and bootstrap argument) is lose effectiveness due to the lack of time-decay of solution for the linear perturbation equation. Meanwhile, a global well-posed result of
linear damped wave with variable time-space coefficients is established. Hence, our result construct a unique global timelike non-small solution near the hyperplane.\\


\end{abstract}



\section{Introduction and main results}

\subsection{Introduction}
Let $\mathcal{M}$ be a timelike $(M+1)$-dimensional hypersurface, and $(\mathbb{R}^{D},g)$ be a $D$-dimensional Minkowski space, and $g$ be the Minkowski metric with $g(\partial_t,\partial_t)=1$. At any time $t$, the spacetime volume in $\mathbb{R}^{D}$ of timelike hypersurface $\mathcal{M}$ can be described as a graph over $\mathbb{R}^M$, which satisfies
\begin{equation}\label{E1-0}
\mathcal{S}(u)=\int_{\mathbb{R}}\int_{\mathbb{R}^{M}}\sqrt{1-|\partial_t u|^2+|\nabla u|^2}d^Mxdt.
\end{equation}
Critical points of action integral (\ref{E1-0}) give rise to submanifolds $\mathcal{M}\subset\mathbb{R}^{D}$ with
vanishing mean curvature, i.e. timelike extremal hypersurfaces. The Euler-Lagrange equation of (\ref{E1-0}) is
\begin{equation}\label{ENNN1-1}
\partial_t\left(\frac{\partial_t u}{\sqrt{1-|\partial_t u|^2+|\nabla u|^2}}\right)-{\rm div}\left(\frac{\nabla u}{\sqrt{1-|\partial_t u|^2+|\nabla u|^2}}\right)=0,
\end{equation}
which admits an exact scaling invariance
$$
u(t,x)\mapsto u_{\lambda}(t,x)=\lambda u(\lambda^{-1}t,\lambda^{-1}x),\quad for \quad any \quad constant \quad \lambda>0,
$$
and it is a mass conservation dynamics, i.e.
$$
\int_{\RR}\left(\frac{\partial_t u}{\sqrt{1-|\partial_t u|^2+|\nabla u|^2}}\right)dx_i~is~conserved~along~the~dynamics,
$$
and one can see that the stationary equation of it is the minimal surface equation
\begin{equation}\label{0ENNN1-1}
{\rm div}\left(\frac{\nabla u}{\sqrt{1+|\nabla u|^2}}\right)=0, \quad \forall x\in\RR^M.
\end{equation}
Bernstein conjecture that the solution of (\ref{0ENNN1-1}) is a linear function in its variables.
 Simons \cite{Simon} proved it is true the for dimension $M\leq 7$.
The famous Bombieri-De Giorgi-Giusti minimal graph \cite{BGG} gives a counter-example for the dimension $M=8$, which also disproves the Bernstein conjecture for all dimension $M\geq8$.
Thus three nature questions are arisen in the timelike extremal hypersurfaces theory:\\

(i) Is the hyperplane as the solution of the minimal surface equation (\ref{0ENNN1-1}) stable for the timelike extremal hypersurfaces equation (\ref{ENNN1-1}) in the dimension $2\leq M\leq 7$?\\

(ii) If the dimension $M\geq8$, are there solutions of the timelike extremal hypersurfaces equation (\ref{ENNN1-1}) convergence to \textbf{non-planar} solutions of the minimal surface equation (\ref{0ENNN1-1}) as the time $t$ goes to $+\infty$?\\

(iii) Is the stationary solution of the timelike extremal hypersurfaces equation (\ref{ENNN1-1}) stable in some function spaces under stochastic perturbations?  \\

The timelike minimal surface equation arises in string theory and geometric minimal surfaces theory in Minkowski space.
The global regularity of this equation with the small intial data has been widely studied,
one can see the related results in \cite{Ba,Kong,Lin,Mi,Yan1} for the related nonlinear wave equations. For the large initial data,
there has been discovered that the behavior of string theory in spacetimes that develop singularities \cite{W}. Meanwhile,
the study of singularity is one of most important topics in physics and mathematics theory, which corresponds to a physical event. It can also imply that some essential physics is missing from the equation in question, which should thus be supplemented with additional terms.
For the classicification of solutions in physics, there are the timlike solution, the spacelike solution and the lightlike (null) solution. To the equation (\ref{E1-1}), if the solution $u(t,x)$ of it satisfies
$1+|\nabla u|^2-u_t^2>0$, then it is called the timelike solution; if the solution $u(t,x)$ of it satisfies
$1+|\nabla u|^2-u_t^2<0$, then it is called the spacelike solution; if it holds
$1+|\nabla u|^2-u_t^2=0$,  then it is called the lightlike solution. Eggers \& Hoppes \cite{Hop1,Hop2} gave a detailed discussion on the existence of lightlike singularity for the Born-Infeld equation
 (or called relavisitive string equation)
$$
u_{tt}(1+u_x^2)-u_{xx}(1-u_t^2)=2u_tu_xu_{tx},\quad (t,x)\in\RR^+\times\RR,
$$
which is one dimensional case of timelike minimal surface equation (\ref{ENNN1-1}).
They showed that it admits lightlike self-similar blowup solutions
\begin{equation*}
u(t,x)=u_0-\hat{t}+\hat{t}^{a}h(\frac{x}{\hat{t}^b})+\ldots,
\end{equation*}
where $\hat{t}=t_0-t$ and $h(x)\varpropto A_{\pm}x^{\frac{2a}{a+1}}$ for $x\rightarrow\pm\infty$. In higher dimension case, they showed that the radially symmetric membranes equation admits self-similar solutions
\begin{equation*}
u(t,x)=-\hat{t}+\hat{t}^{a}h(\frac{x-x_0}{\hat{t}^b})+\ldots,
\end{equation*}
by analyzing the eikonal equation
\begin{equation*}
1-u_t^2+u_x^2=0.
\end{equation*}
Meanwhile, the swallowtail singularity was also been given by parametric string solution in \cite{Egg}.
Yan \cite {Yan1} found that both the Born-Infeld equation and the linear wave equation admit the same family of explicit self-similar solutions
$$
u_k(t,x)=k\ln(\frac{T-t+x}{T-t-x}),\quad |x|< T-t,\quad t\in[0,T),\quad \forall k\in\RR/\{0\},
$$
where $T$ denotes the maximal existence time.
Moreover, those explicit solutions of the Born-Infeld equation are the timelike singularities.
In two dimension case, Nguyen \& Tian \cite{tian} proved the existence of blowup solution when the string moving in Einstein vaccum spacetime. After that,
Yan \cite{Yan2} considered the the radially symmetric case:
$$
u_{tt}-u_{rr}-\frac{u_r}{r}+u_{tt}u_r^2+u_{rr}u_t^2-2u_tu_ru_{tr}+\frac{1}{r}u_ru_t^2-\frac{1}{r}u_r^3=0,
$$
where $r=|x|$, it admits two stable explicit lightlike self-similar solutions
$$
u_T^{\pm}(t,r)=\pm(T-t)\sqrt{1-(\frac{r}{T-t})^2}, \quad\quad (t,r)\in(0,T)\times[0,T-t],
$$
where the positive constant $T$ denotes the maximal existence time.
In spired by the radially symmetric case, one can check the timlike minimal surface equation (\ref{E1-1}) admits two explicit lightlike blowup solutions:
\bel{Yn01-1}
u_T^{\pm}(t,x)=\pm(T-t)\Big(1-\sum_{i=1}^M{x_i^2\over (T-t)^2}\Big)^{{1\over2}},
\ee
which are two self-similar spheres in geometry. Those two functions exhibit the smooth for all $0<t<T$, but which break down at $t=T$ in the sense that
$$
\partial_{x_1x_1}u_T^{\pm}(t,x)|_{x_1=0}\rightarrow+\infty,~~as~~t\rightarrow T^{-},
$$
and the dynamical behavior of them are as attractors.
At the initial time $t=0$, the form of it is a sphere:
$$
x_1^2+x_2^2+\ldots+x_n^2+(u_T^{\pm}(0,x))^2=T^2,
$$
then as the time $t$ approach the blowup time $T$, it begins to expand until it starts to shrink and eventually collapses to a point at the time $T$.
In $4$-dimensional radial case, Bahouri \& Perelman \& Marachli constructed a blowup solution of
the hyperbolic vanishing mean curvature flow surfaces asymptotic to Simons cone. Recently, Yan \cite{YN} showed that (\ref{ENNN1-1}) admits the stable self-similar shrinkers (the shape of them are spheres) without the radial assumption in higher dimension $M\geq 9$. It is still open question for the dimension $2\leq M\leq 8$.

\subsection{Main result}
In the present paper, we aim to  investigate the dynamical behavior around the stationary solution of equation (\ref{ENNN1-1}) (i.e. answer the problem (i)).
We supplement equation (\ref{ENNN1-1}) with an initial data
\begin{equation}\label{E1-2R1}
u(0,x)=u_0(x),~~u_{t}(0,x)=u_1(x).
\end{equation}
Since the minimal surface equation (\ref{0ENNN1-1}) only admits the hyperplane solution, we can denote it by
\begin{equation}\label{E01-2R1}
u_s(x)=A\cdot x+B,~~\forall x\in\RR^M,~~2\leq M\leq7,
\end{equation}
where $A$ and $B$ denote two constant vectors in $\RR^M$, and we require that $A$ is non-zero vector. According to the classification of solution, (\ref{E01-2R1}) is the spacelike solution of the timelike minimal surface equation (\ref{ENNN1-1}).

We state the main result.
\begin{theorem}
Let the dimension $2\leq M\leq 7$. The hyperplane solution (\ref{E01-2R1}) of
the timelike extremal hypersurfaces equation (\ref{ENNN1-1}) is asymptotic stable in Sobolev space $H^s(\RR^M)$ for any $s\geq1$, that is,
there exist a small positive constant $\eps$, if the initial data (\ref{E1-2R1}) satisfies
$$
u_0(x)=u_s(x)+w_0(x),\quad u_1(x)=w_1(x),
$$
where smooth functions $w_0(x)$ and $w_1(x)$ are supported in $\Big\{x\in\RR^M   \Big ||x|\leq1\Big\}$, and
$$
\|w_0(x)\|_{H^{s+1}(\RR^M)}+\|w_1(x)\|_{H^s(\RR^M)}<\eps,
$$
then the timelike extremal hypersurfaces equation (\ref{ENNN1-1}) admits a unique global solution $u(t,x)\in H^s(\RR^M)$ such that
$$
\lim_{t\rightarrow+\infty}\|u(t,x)-u_s(x)\|_{H^s(\RR^M)}=0.
$$
\end{theorem}

\begin{remark}
We remark that our result also build up a global well-posedness result for the the timelike extremal hypersurfaces equation (\ref{ENNN1-1}). 
The uniqueness global solution takes the form
$$
u(t,x)=u_s(x)+w(t,x),
$$
with 
$$
\sup_{t\in(0,\infty)}\|w(t,x)\|_{H^s(\RR^M)}\lesssim\eps.
$$ 
Meanwhile, $w(t,x)$ decays in time with \textbf{polynomial} form.

Furthermore, it holds
\begin{eqnarray*}
1+|\nabla u|^2-u_t^2&=&1+|\nabla u_s(x)+\nabla w|^2-|w_t|^2\\
&\sim&1+|\nabla u_s(x)|^2+O(\eps)>0.
\end{eqnarray*}
Therefore, we construct a unique global timelike non-small solution near the hyperplane.

\end{remark}

\begin{remark}
For the dimension $M\geq8$, theorem 1.1 also holds. But we are more interested in the stability of non-hyperplane solution of minimal surface equation for the dimension $M\geq8$,
for example, the stability of Simons cone.
\end{remark}

\subsection{Sketch of the proof}

Equation (\ref{ENNN1-1}) can be rewritten as
$$
\Big(1+|\nabla_x u|^2-u_t^2\Big)^{-{3\over2}}\Fcal(u)=0,\quad \forall (t,x)\in\RR^+\times\RR^M,
$$
where
\begin{eqnarray}\label{n1-0}
\Fcal(u)&:=&
u_{tt}-\Delta u+u_{tt}|\nabla u|^2-\frac{1}{2}u_t\partial_t|\nabla u|^2\nonumber\\
&&-\sum_{k=1}^M\left(\partial^2_{x_k}u(|\nabla u|^2-u_t^2)-\frac{1}{2}\partial_{x_k}u(\partial_{x_k}|\nabla u|^2-\partial_{x_k}u_t^2)\right).
\end{eqnarray}
Thus finding the solution of (\ref{ENNN1-1}) is equivalent to solve the equation
\begin{equation}\label{E1-1}
(1+|\nabla u|^2)u_{tt}-(1+|\nabla u|^2-u_t^2)\Delta u-\frac{1}{2}u_t\partial_t|\nabla u|^2+\frac{1}{2}\sum_{k=1}^M\partial_{x_k}u\partial_{x_k}(|\nabla u|^2-u_t^2)=0,~~
\end{equation}
where $ \forall (t,x)\in\RR^+\times\RR^M$, and the operator $\Delta=\sum_{k=1}^M\del^2_{x_k}$ is the Laplace-Beltrami operator.

From (\ref{E1-1}), the stationary equation of it is
\bel{y1-0}
-(1+|\nabla u|^2)\Delta u+\frac{1}{2}\sum_{k=1}^M\partial_{x_k}u\partial_{x_k}|\nabla u|^2=0,
\ee
which is equivalent to the minimal surface equation (\ref{0ENNN1-1}). Thus to prove theorem 1.1, we only need to consider the perturbation equation from (\ref{E1-1})-(\ref{y1-0}).
More precisely, we set the solution of (\ref{E1-1}) having the form
$$
u(t,x)=u_s(t,x)+w(t,x), \quad \forall x\in\RR^M,~t>0,
$$
where $u_s(t,x)$ given in (\ref{E01-2R1}) is the solution of (\ref{y1-0}), then we substitute it into (\ref{E1-1}) to get the perturbation equation:
\bel{E0001-1}
\aligned
\Lcal w:=&\Big(1+|A+\nabla w|^2\Big)w_{tt}-\Big(1+|A+\nabla w|^2-w_t^2\Big)\Delta w-\frac{1}{2}w_t\partial_t\Big((\nabla w+2A)\cdot\nabla w\Big)\\
&+\frac{1}{2}\sum_{k=1}^M\Big(\partial_{x_k}w+A_k\Big)\partial_{x_k}\Big((\nabla w+2A)\cdot\nabla w-w_t^2\Big)=0,
\endaligned
\ee
which is a quasilinear wave equation, and the linear equation of it is
\bel{yy0-1}
w_{tt}-\Delta w+(1+|A|^2)^{-1}\sum_{k=1}^M\sum_{k'=1}^MA_kA_{k'}\del_{x_k}\del_{x_{k'}}w=0.
\ee
We supplement it with an initial data
$$
w(t,x)|_{t=0}=w_0(x),\quad w_t(t,x)|_{t=0}=w_1(x),
$$
then utilizing the energy inequality given in Proposition 6.3.2 of the book of H\"{o}rmander \cite{H1}, it holds
$$
\|w_t\|_{L^2(\RR^M)}+\sum_{k=1}^M\|\del_{x_k}w\|_{L^2(\RR^M)}\leq 2\Big(\|w_1(x)\|_{L^2(\RR^M)}+\sum_{k=1}^M\|\del_{x_k}w(t,x)|_{t=0}\|_{L^2(\RR^M)}\Big),
$$
from which, the bootstrap argument (based on the energy estimate) loses efficacy due to
the absence of time-decay of solution for the linear wave equation (\ref{yy0-1}). But if we choose a function $w^{(0)}(t,x)$ to be satisfies
$$
-(A+\nabla w^{(0)})\cdot\nabla w^{(0)}_t+w^{(0)}_t\Delta w^{(0)}>0,\quad \forall t>T^*>0,
$$
then
we linearize nonlinear equation (\ref{E0001-1}) at it to get the following linear damped wave equation
$$
\aligned
\Lcal^{(0)}_{w^{(0)}}h:=&\Big(1+|A+\nabla w^{(0)}|^2\Big)h_{tt}-\Big(1+|A+\nabla w^{(0)}|^2-(w^{(0)}_t)^2\Big)\Delta h\\
&+\sum_{k=1}^M\Big(\del_{x_k}w^{(0)}+A_k\Big)\Big(\nabla w^{(0)}+A\Big)\cdot\nabla\del_{x_k}h+2\Big(-(A+\nabla w^{(0)})\cdot\nabla w^{(0)}_t+w^{(0)}_t\Delta w^{(0)}\Big)h_t\\\
&+\Big[{1\over2}\nabla\Big(|\nabla w^{(0)}|^2-2(w^{(0)}_t)^2+2A\cdot\nabla w^{(0)}\Big)+\sum_{k=1}^M\Big(\del_{x_k}w^{(0)}+A_k\Big)\nabla\del_{x_k}w^{(0)}\\
&+2\Big(\nabla w^{(0)}+A\Big)\Big(w^{(0)}_{tt}-\Delta w^{(0)}\Big)\Big]\cdot\nabla h-2w^{(0)}_t\Big(\nabla w^{(0)}+A\Big)\cdot\nabla h_t=0,\quad t>T^*>0,
\endaligned
$$
where the positive constant $T^*$ is the local existence time of solution for above linear wave equation. It is not a damped linear wave equation for $t\in[0,T^*]$. 
The proof of local existence of it can be followed from the book of Sogge \cite{Sog}.

It gives a possible way to get a time-decay solution of linear wave equation. But the function $w^{(0)}(t,x)$ is not a solution of equation (\ref{E0001-1}), there must be an error term denoted by
$$
E^{(0)}:=\Lcal(w^{(0)}),
$$
meanwhile, the function $w^{(0)}(t,x)$ should be chosen to make the error term small, i.e.
$$
E^{(0)}=\Lcal(w^{(0)})\sim\eps,~~in~some~function~space,
$$
where $E^{(0)}$ is called as the initial error term. In order to construct the solution of nonlinear equation (\ref{E0001-1}), we should approximate it step by step. So the first approximation solution has the form
$w^{(1)}(t,x):=w^{(0)}(t,x)+h^{(1)}(t,x)$, where $h^{(1)}(t,x)$ is the solution of linear damped wave equation
$$
\Lcal^{(0)}_{w^{(0)}}h^{(1)}=E^{(0)}.
$$
Forward this idea, we get the $m$th approximation step $h^{(m)}(t,x)$ by solving the linear equation
$$
\Lcal^{(0)}_{w^{(m-1)}}h^{(m)}=E^{(m-1)},\quad \forall m\in\NN,
$$
where the error term $E^{(m-1)}:=\Lcal(w^{(m-1)})$. Then the $m$th approximation solution is obtained as the form
$$
w^{(m)}(t,x)=w^{(0)}(t,x)+\sum_{i=1}^mh^{(i)}(t,x).
$$
At last, the most important thing is to prove
$$
\lim_{m\rightarrow+\infty}w^{(m)}(t,x)=w^{(\infty)}(t,x)<+\infty,
$$
and the error term
$$
\lim_{m\rightarrow+\infty}E^{(m)}(t,x)=\lim_{m\rightarrow+\infty}\Lcal(w^{(\infty)})=0.
$$
We mention that the paper of Yang \cite{Yang} proved the global well-posedness for a class of nonlinear wave equation with variable coefficients when the nonlinear term satisfies the null condition.
Here we should notice that there is loss of derivatives in each iteration step due to the quasilinear terms in (\ref{E0001-1}), so
we have to use the smooth operator (see \cite{AP} for more details on this operator) to smooth the linearized equation at each iteration step.
Therefore, we construct the solution $w^{(\infty)}(t,x)$ of nonlinear equation (\ref{E0001-1}).
Above method is called as Nash-Moser iteration scheme.
It has been used in \cite{Yan,Yan1,Yan2,YZ}.
We refer the readers to \cite{H,H1,Moser,Nash,R} for more details of this method.

\textbf{Notation.} Thoughout this paper,
we denote
$\mathbb{N}$ by the natural numbers $\{1,2,3,\ldots\}$. $\textbf{0}$ is the vector of zero.
The symbol $a\lesssim b$ means that there exists a positive constant $C$ such that $a\leq Cb$.
$\CC^{\infty}_0(\RR^+\times\RR^M)$ is the space of $u:\RR^+\times\RR^M\rightarrow\RR$, and $u$ is infinitely differentiable with compact support.
Furthermore, we denote the usual norm of Sobolev space $H^l(\RR^M)$ by $\|\cdot\|_{H^l}$ for convenience.
The space $\LL^2((0,\infty);H^l(\RR^M)$ is equipped with the norm
$$
\|v\|^2_{\LL^2((0,\infty);H^l)}:=\int_0^{\infty}\|v(t,\cdot)\|^2_{H^l}dt.
$$

The organization of this paper is as follows. In Section 2, we give a general global existence result of a class of linear damped wave equation with variable coefficients.
 In Section 3, the well-posedness of linearized problem is shown by finding the time decay estimate of first approximation step.
After that, we show the existence of general approximation step for nonlinear perturbation problem. In the last section, the convergence of approximation scheme is given.



\section{The linear damped wave equation with variable coefficients }\setcounter{equation}{0}

In this section, we give the general existence of result  for a class of linear damped wave equation with \textbf{smooth} variable coefficients.
We consider the following initial value problem:
\bel{NYN0-1}
\aligned
A(t,x)h_{tt}&-B(t,x)\Delta h+C(t,x)h_t+\sum_{k=1}^MD_k(t,x)\partial_{x_k}h+\sum_{k=1}^ME_k(t,x)\partial_{x_k}h_t\\
&+\sum_{k=1}^M\sum_{i=1}^MH_{ki}(t,x)\partial_{x_k}\partial_{x_i}h=f(t,x),\quad \forall (t,x)\in\RR^+\times\RR^M,
\endaligned
\ee
with an initial data
$$
h(0,x)=h_0(x),\quad h_t(0,x)=h_1(x),\quad \forall x\in\RR^M.
$$
We assume that coefficients of (\ref{NYN0-1}) satisfy the following condition:
$$
A(t,x), B(t,x), C(t,x), D_k(t,x), E_k(t,x), H_{ki}(t,x)\in\CC^{\infty}(\RR^M),\quad \forall k,i=1,\ldots,M,
$$
and
\bel{NYN0-2}
\sigma_0>A(t,x)>\sigma>1,\quad B(t,x)>\sigma>1,\quad C(t,x)>0, \quad\forall ( t,x)\in\RR^+\times\RR^M,
\ee
and $H_{ki}=H_{ik}$,
\bel{NYN0-3}
|H|:=\sum_{k,i=1}^M|H_{ki}|\leq\sigma\quad and \quad H_{ki}>0,
\ee
and
\bel{NYN0-4}
|D|:=\sum_{k=1}^M|D_{k}|\sim\eps,\quad |E|:=\sum_{k=1}^M|E_{k}|\sim\eps,
\ee
and
\bel{NYN0-04}
\aligned
&\|\del^sA\|_{\LL^{\infty}}\sim\eps,\quad \|\del^sB\|_{\LL^{\infty}}\sim\eps,\quad \|\del^sC\|_{\LL^{\infty}}\sim\eps,\quad \|\del^sD_k\|_{\LL^{\infty}}\sim\eps,\\
&\|\del^sE_k\|_{\LL^{\infty}}\sim\eps,\quad \|\del^sH_{ik}\|_{\LL^{\infty}}\sim\eps,\quad\forall s\in\{1,2,3,\ldots\},
\endaligned
\ee
for a positive small constant $\eps$. Here we use $\del$ to denote the derivative of time or spacial variable.

We choose two weighted positive smooth functions $\varphi(t,x)$ and $\overline{\varphi}(t,x)$ in $\CC^{\infty}_0(\RR^+\times\RR^M)$, and satisfying
\begin{eqnarray}
\label{NYN0-5}
&&\varphi_t+c\varphi\leq0,\quad \partial^s\varphi\sim\eps,
\\
\label{NYN0-7}
&&\overline{\varphi}_t+c\overline{\varphi}\leq0,\quad \overline{\varphi}_{tt}\geq c^2\overline{\varphi},\quad \partial^s\overline{\varphi}\sim\eps\\
\label{NYN0-6}
&&c\overline{\varphi}\leq \varphi,
\end{eqnarray}
with positive constant $c$. Moreover, there exists a positive constant $C_{c,\sigma,\eps}$ depending on
parameters $\sigma$ and $c$ such that
\bel{NYN0-06}
\overline{\varphi}^{-1}e^{-C_{c,\sigma,\eps}t}\leq e^{-\eps t},\quad \forall t>0.
\ee
Here the value of $\sigma$ is crucial for above assumption.

We now derive a weigthed $\LL^2$-estimate of solution for the linear equation (\ref{NYN0-1}).

\begin{lemma}
Let $f\in \CC((0,\infty); \LL^2(\RR^M))$. Assume that (\ref{NYN0-2})-(\ref{NYN0-4}) hold.
Then the solution of linear wave equation (\ref{NYN0-1}) satisfies
$$
\int_{\RR^M}\overline{\varphi}\Big((h_t)^2
+|\nabla h|^2+ h^2\Big)dx\lesssim e^{-C_{c,\sigma,\eps}t}\Big[\int_{\RR^M}\overline{\varphi}(0,x)\Big(h_1^2
+|\nabla h_0|^2+ h_0^2\Big)dx+\int_0^{\infty}\int_{\RR^M}\overline{\varphi}f^2dxdt\Big].
$$
\end{lemma}
\begin{proof}
On one hand, we multiply equation \eqref{NYN0-1} with $\varphi(t,x)h_t$, then integrating it over $\RR^M$ on $x$, it holds
\begin{eqnarray}\label{2.1}
&&\int_{\RR^M}A\varphi h_{tt}h_tdx-\int_{\RR^M}B\varphi\Delta hh_tdx
+\int_{\RR^M}C\varphi(h_{t})^2dx
+\sum_{k=1}^M\int_{\RR^M}\varphi D_k\partial_{x_k}hh_tdx\nonumber\\
&&+\sum_{k=1}^M \int_{\RR^M}\varphi E_k\partial_{x_k}h_th_tdx
+\sum_{k,i=1}^M\int_{\RR^M} \varphi H_{ki}\partial_{x_k}\partial_{x_i}hh_tdx=\int_{\RR^M}f\varphi h_tdx.
\end{eqnarray}
Direct computation gives that
\begin{equation}\label{A1}
\int_{\RR^M}A\varphi h_{tt}h_tdx=\frac{1}{2}\frac{d}{dt}\int_{\RR^M}A\varphi(h_t)^2dx-\frac{1}{2}\int_{\RR^M}\partial_t(A\varphi)(h_t)^2dx,
\end{equation}
and
\begin{eqnarray}\label{B1}
-\int_{\RR^M}B\varphi\Delta hh_tdx
&=&\sum_{k=1}^M\int_{\RR^M}\partial_{x_k}(B\varphi)\partial_{x_k}hh_tdx
+\sum_{k=1}^M\int_{\RR^M}B\varphi\partial_{x_k}h\partial_{x_k}h_tdx\nonumber\\
&=&\sum_{k=1}^M\int_{\RR^M}\partial_{x_k}(B\varphi)\partial_{x_k}h  h_tdx
+\frac{1}{2}\frac{d}{dt}\sum_{k=1}^M\int_{\RR^M}B\varphi(\partial_{x_k}h)^2dx\nonumber\\
&&-\frac{1}{2}\sum_{k=1}^M\int_{\RR^M}\del_t(B\varphi)(\partial_{x_k}h)^2,
\end{eqnarray}
and
\begin{equation}\label{E1}
\sum_{k=1}^M\int_{\RR^M}\varphi E_k\partial_{x_k}h_th_tdx
=-\frac{1}{2}\sum_{k=1}^M\int_{\RR^M}\partial_{x_k}(E_k\varphi)(h_t)^2dx,
\end{equation}
and
\begin{eqnarray}\label{H1}
\sum_{k,i=1}^M\int_{\RR^M}\varphi H_{ki}\partial_{x_k}\partial_{x_i}hh_tdx
=-\frac{1}{2}\frac{d}{dt}\sum_{i,l=1}^M\int_{\RR^M}\varphi H_{il}\partial_{x_i}h\partial_{x_l}hdx,
\end{eqnarray}
where equality \eqref{H1} is derived by utilizing the following formula
$$
2\sum_{i=0}^nK^i\partial_iu\sum_{j,k=0}^ng^{jk}\partial_{j}\partial_{k}u
=\sum_{i,j=0}^n\partial_{j}(T_i^j(u)K^i),
$$
with
$$
 T_i^j(u)=2\sum_{k=0}^ng^{jk}\partial_{k}u\partial_{i}u-\delta_i^j\sum_{k,l=0}^ng^{kl}\partial_{k}u\partial_{l}u,
$$
by setting $H_{00}=H_{k0}=H_{0i}=0$. $\delta_i^j=1$ for $i=j$, otherwise, it is zero. One can see page 97 in the book of H\"{o}rmander \cite{H1} for more details of above formula.

 Thus,  by \eqref{A1}-\eqref{H1}, we reduce \eqref{2.1} into
\begin{eqnarray}\label{T1}
&&\frac{1}{2}\frac{d}{dt}\int_{\RR^M}\Big[A\varphi(h_t)^2
+B\varphi\sum_{k=1}^M(\partial_{x_k}h)^2
-\varphi\sum_{i,l=1}^MH_{il}\partial_{x_i}h\partial_{x_l}h\Big]dx\nonumber\\
&&\quad+\frac{1}{2}\int_{\RR^M}\Big[-\del_t(A\varphi)
+2C\varphi-\sum_{k=1}^M\partial_{x_k}(E_k\varphi)\Big](h_t)^2dx\nonumber\\
&&\quad-\frac{1}{2}\sum_{k=1}^M\int_{\RR^M}\del_t(B\varphi)(\partial_{x_k}h)^2dx
+\sum_{k=1}^M\int_{\RR^M}\varphi D_k\partial_{x_k}hh_tdx+\sum_{k=1}^M\int_{\RR^M}\partial_{x_k}(B\varphi)\partial_{x_k}h  h_tdx\nonumber\\
&&=\int_{\RR^M}f\varphi h_tdx.
\end{eqnarray}

On the other hand, we multiply equation \eqref{NYN0-1} with $\overline{\varphi}(t,x)h$,  then integrating it over $\RR^M$ on $x$, it holds
\begin{eqnarray}\label{2.2}
&&\int_{\RR^M}A\overline{\varphi}h_{tt}hdx
-\int_{\RR^M}B\overline{\varphi}\Delta hhdx
+\int_{\RR^M}C\overline{\varphi}h_{t}hdx+\sum_{k=1}^M\int_{\RR^M}\overline{\varphi}D_k\partial_{x_k}hhdx\nonumber\\
&&\quad+\sum_{k=1}^M\int_{\RR^M}\overline{\varphi}E_k\partial_{x_k}h_thdx
+\sum_{k,i=1}^M\int_{\RR^M}\overline{\varphi}H_{ki}\partial_{x_k}\partial_{x_i}hhdx\nonumber\\
&&=\int_{\RR^M}f\overline{\varphi}hdx.
\end{eqnarray}
We notice that
\begin{eqnarray}\label{A2}
\int_{\RR^M}A\overline{\varphi}h_{tt}hdx
&=&\frac{d}{dt}\int_{\RR^M}A\overline{\varphi}h_t hdx-\int_{\RR^M}\pdt(A\overline{\varphi})h_t hdx-\int_{\RR^M}A\overline{\varphi}(h_t)^2dx\nonumber\\
&=&\frac{d}{dt}\int_{\RR^M}\Big(A\overline{\varphi}h_t h
-\frac{1}{2}\pdt(A\overline{\varphi})h^2\Big)dx
+\frac{1}{2}\int_{\RR^M}\pdt^2(A\overline{\varphi})h^2dx\nonumber\\
&&\quad-\int_{\RR^M}A\overline{\varphi}(h_t)^2dx,~~~~~~~~~
\end{eqnarray}
and
\begin{eqnarray}\label{B2}
-\int_{\RR^M}B\overline{\varphi}\Delta hhdx
&=&\sum_{k=1}^M\int_{\RR^M}\partial_{x_k}(B\overline{\varphi})\partial_{x_k}h hdx
+\sum_{k=1}^M\int_{\RR^M}B\overline{\varphi}(\partial_{x_k}h)^2dx\nonumber\\
&=&-\frac{1}{2}\sum_{k=1}^M\int_{\RR^M}\partial_{x_k}^2(B\overline{\varphi})h^2dx
+\sum_{k=1}^M\int_{\RR^M}B\overline{\varphi}(\partial_{x_k}h)^2dx,~~~~~~~~
\end{eqnarray}
and
\begin{eqnarray}\label{C2}
\int_{\RR^M}C\overline{\varphi}h_thdx
&=&\frac{1}{2}\frac{d}{dt}\int_{\RR^M}C\overline{\varphi}h^2dx-\frac{1}{2}\int_{\RR^M}\pdt(C\overline{\varphi})h^2dx,\\
\sum_{k=1}^M\int_{\RR^M}D_k\overline{\varphi}\partial_{x_k}hhdx
&=&-\frac{1}{2}\sum_{k=1}^M\int_{\RR^M}\partial_{x_k}(D_k\overline{\varphi})h^2dx,
\end{eqnarray}
and
\begin{eqnarray}\label{E2}
\sum_{k=1}^M\int_{\RR^M}E_k\overline{\varphi}\partial_{x_k}h_thdx
&=&-\sum_{k=1}^M\int_{\RR^M}\partial_{x_k}(E_k\overline{\varphi})h_t  hdx
-\sum_{k=1}^M\int_{\RR^M}E_k\overline{\varphi}h_t\partial_{x_k} hdx\nonumber\\
&=&-\frac{1}{2}\frac{d}{dt}\sum_{k=1}^M\int_{\RR^M}\partial_{x_k}(E_k\overline{\varphi})h^2dx
+\frac{1}{2}\sum_{k=1}^M\int_{\RR^M}\pdt(\partial_{x_k}(E_k\overline{\varphi}))h^2dx\nonumber\\
&&-\sum_{k=1}^M\int_{\RR^M}E_k\overline{\varphi}h_t\partial_{x_k} hdx,
\end{eqnarray}
and
\begin{eqnarray}\label{H2}
\sum_{k,i=1}^M\int_{\RR^M}H_{ki}\overline{\varphi}\partial_{x_k}\partial_{x_i}h h dx
&=&-\sum_{k,i=1}^M\int_{\RR^M}\partial_{x_k}(H_{ki}\overline{\varphi})\partial_{x_i}h hdx
-\sum_{k,i=1}^M\int_{\RR^M}H_{ki}\overline{\varphi}\partial_{x_i}h \partial_{x_k}hdx\nonumber\\
&=&\frac{1}{2}\sum_{k,i=1}^M\int_{\RR^M}\partial_{x_i}\partial_{x_k}(H_{ki}\overline{\varphi})h^2dx
-\sum_{k,i=1}^M\int_{\RR^M}H_{ki}\overline{\varphi}\partial_{x_i}h \partial_{x_k}hdx.~~~~~~~~
\end{eqnarray}

So we combine \eqref{A2}-\eqref{H2} with \eqref{2.2} to get
\begin{eqnarray}\label{T2}
&&\frac{1}{2}\frac{d}{dt}\int_{\RR^M}\Big[\Big(-\pdt(A\overline{\varphi})+C\overline{\varphi}-\sum_{k=1}^M\partial_{x_k}(E_k\overline{\varphi})\Big)h^2
+2A\overline{\varphi}h_th\Big]dx\nonumber\\
&&\quad+\frac{1}{2}\int_{\RR^M}\Big[\pdt^2(A\overline{\varphi})-\sum_{k=1}^M\partial_{x_k}^2(B\overline{\varphi})-\pdt(C\overline{\varphi})
+\sum_{k,i=1}^M\partial_{x_i}\partial_{x_k}(H_{ki}\overline{\varphi})\nonumber\\
&&\quad+\sum_{k=1}^M\partial_{x_k}\del_t(E_k\overline{\varphi})-\sum_{k=1}^M\partial_{x_k}(D_k\overline{\varphi})\Big]h^2dx\nonumber\\
&&\quad+\sum_{k=1}^M\int_{\RR^M}B\overline{\varphi}(\partial_{x_k}h)^2dx-\int_{\RR^M}A\overline{\varphi}(h_t)^2dx
-\sum_{k,i=1}^M\int_{\RR^M}H_{ki}\overline{\varphi}\partial_{x_i}h\partial_{x_k}hdx\nonumber\\
&&\quad-\sum_{k=1}^M\int_{\RR^M}E_k\overline{\varphi}h_t\partial_{x_k} hdx\nonumber\\
&&=\int_{\RR^M}f\overline{\varphi}hdx.
\end{eqnarray}

Furthermore, we use Cauchy inequality to derive
\begin{eqnarray}
\label{w1}
\sum_{i,l=1}^M\int_{\RR^M}\varphi H_{il}\partial_{x_i}h\partial_{x_l}hdx&\leq&C_M\sum_{k=1}^M\int_{\RR^M}\varphi|H|(\partial_{x_k}h)^2dx,\\
\label{w2}
\sum_{k=1}^M\int_{\RR^M}\varphi D_k\partial_{x_k}hh_tdx&\leq&{1\over2}\int_{\RR^M}\varphi|D|\Big(\sum_{k=1}^M(\partial_{x_k}h)^2+(h_t)^2\Big)dx,\\
\label{w003}
\sum_{k=1}^M\int_{\RR^M}\partial_{x_k}(B\varphi)\partial_{x_k}h  h_tdx&\leq&{1\over2}\sum_{k=1}^M\int_{\RR^M}|\partial_{x_k}(B\varphi)|\Big((\partial_{x_k}h)^2+(h_t)^2\Big)dx,\\
\label{w3}
\int_{\RR^M}f\varphi h_tdx&\leq&{1\over2}\int_{\RR^M}\varphi f^2dx+{1\over2}\int_{\RR^M}\varphi (h_t)^2dx,
\end{eqnarray}
and
\begin{eqnarray}\label{w4}
2\int_{\RR^M}A\overline{\varphi}h_th dx&\leq&\int_{\RR^M}A\overline{\varphi}((h_t)^2+h^2)dx,\quad for~A>0,\\
\label{w5}
-\sum_{k,i=1}^M\int_{\RR^M}H_{ki}\overline{\varphi}\partial_{x_i}h\partial_{x_k}hdx&\leq&C_M\sum_{k=1}^M\int_{\RR^M}|H|\overline{\varphi}(\partial_{x_k}h)^2dx,\\
\label{w6}
-\sum_{k=1}^M\int_{\RR^M}E_k\overline{\varphi}h_t\partial_{x_k} hdx&\leq&{1\over2} \int_{\RR^M}|E|\overline{\varphi}(h_t)^2dx+{1\over2}\sum_{k=1}^M\int_{\RR^M}|E|\overline{\varphi}(\partial_{x_k} h)^2dx,\\
\label{w7}
\int_{\RR^M}f\overline{\varphi}hdx&\leq&{1\over2}\int_{\RR^M}\overline{\varphi} f^2dx+{1\over2}\int_{\RR^M}\overline{\varphi}h^2dx.
\end{eqnarray}

Hence, by noticing (\ref{w1})-(\ref{w3}) and (\ref{w4})-(\ref{w7}), it follows from \eqref{T1} and \eqref{T2} that
\begin{eqnarray}\label{TT1}
&&\frac{1}{2}\frac{d}{dt}\int_{\RR^M}\Big[A(\varphi-\overline{\varphi})(h_t)^2
+\varphi(B-|H|)|\nabla h|^2\nonumber\\
&&\quad+\Big(-\pdt(A\overline{\varphi})+(C-A)\overline{\varphi}-\sum_{k=1}^M\partial_{x_k}(E_k\overline{\varphi})\Big)h^2
\Big]dx\nonumber\\
&&\quad+\frac{1}{2}\int_{\RR^M}\Big[\pdt^2(A\overline{\varphi})-\Delta(B\overline{\varphi})-\pdt(C\overline{\varphi})+\sum_{k,i=1}^M\partial_{x_i}\partial_{x_k}(H_{ki}\overline{\varphi})\nonumber\\
&&\quad+\sum_{k=1}^M\partial_{x_k}\del_t(E_k\overline{\varphi})-\sum_{k=1}^M\partial_{x_k}(D_k\overline{\varphi})-\overline{\varphi}\Big]h^2dx\nonumber\\
&&\quad+\frac{1}{2}\int_{\RR^M}\Big[-\del_t(A\varphi)
-\sum_{k=1}^M\partial_{x_k}(E_k\varphi)+(2C-|D|-1)\varphi-(2A+|E|)\overline{\varphi}\nonumber\\
&&\quad-\sum_{k=1}^M|\partial_{x_k}(B\varphi)|\Big](h_t)^2dx\nonumber\\
&&\quad+\frac{1}{2}\int_{\RR^M}\Big[(2B-2|H|-|E|)\overline{\varphi}-\del_t(B\varphi)-\varphi|D|-\sum_{k=1}^M|\partial_{x_k}(B\varphi)|\Big]|\nabla h|^2dx
\nonumber\\
&&\lesssim{1\over2}\int_{\RR^M}(\varphi+\overline{\varphi})f^2dx.
\end{eqnarray}

Now we analyse coefficients of inequality (\ref{TT1}). By the assumption given in (\ref{NYN0-2})-(\ref{NYN0-3}) and (\ref{NYN0-6}), it holds
$$
\aligned
&A(\varphi-\overline{\varphi})>c\overline{\varphi} A>c\sigma\overline{\varphi},\\
&\varphi(B-|H|)>\varphi(\sigma-{1\over2}),
\endaligned
$$
and by (\ref{NYN0-04}) and (\ref{NYN0-7}), we have
$$
\aligned
&-\pdt(A\overline{\varphi})+(C-A)\overline{\varphi}-\sum_{k=1}^M\partial_{x_k}(E_k\overline{\varphi})\\
&\geq \Big(A(c-1)-A_t+C\Big)\overline{\varphi}-\sum_{k=1}^M\partial_{x_k}(E_k\overline{\varphi})\\
&\geq \Big(\sigma(c-1)-3\eps\Big)\overline{\varphi}>0,
\endaligned
$$
thus it holds
\begin{eqnarray}
\label{yw1-1}
&&\int_{\RR^M}\Big[A(\varphi-\overline{\varphi})(h_t)^2
+\varphi(B-|H|)|\nabla h|^2+\Big(-\pdt(A\overline{\varphi})+(C-A)\overline{\varphi}-\sum_{k=1}^M\partial_{x_k}(E_k\overline{\varphi})\Big)h^2
\Big]dx\nonumber\\
&&\gtrsim\int_{\RR^M}\Big[c\sigma\overline{\varphi}(h_t)^2
+\varphi(\sigma-{1\over2})|\nabla h|^2+\overline{\varphi}\Big(\sigma(c-1)-3\eps\Big) h^2
\Big]dx.
\end{eqnarray}

Similarly, using (\ref{NYN0-2})-(\ref{NYN0-6}), there exists a positive constant $c_0$ such that
\begin{eqnarray*}
\pdt^2(A\overline{\varphi})&-&\Delta(B\overline{\varphi})-\pdt(C\overline{\varphi})+\sum_{k,i=1}^M\partial_{x_i}\partial_{x_k}(H_{ki}\overline{\varphi})+\sum_{k=1}^M\partial_{x_k}\del_t(E_k\overline{\varphi})-\sum_{k=1}^M\partial_{x_k}(D_k\overline{\varphi})-\overline{\varphi}\nonumber\\
&\geq&A\overline{\varphi}_{tt}-B\Delta\overline{\varphi}-C\overline{\varphi}_t-c_0\eps\overline{\varphi}\nonumber\\
&\geq&\Big(c^2\sigma-c_0(\sigma+1)\eps\Big)\overline{\varphi}>0,
\end{eqnarray*}
and
\begin{eqnarray*}
-\del_t(A\varphi)
&-&\sum_{k=1}^M\partial_{x_k}(E_k\varphi)+(2C-|D|-1)\varphi-(2A+|E|)\overline{\varphi}-\sum_{k=1}^M|\partial_{x_k}(B\varphi)|\nonumber\\
&\geq&A(-\varphi_t-2\overline{\varphi})+(2C-1-5\eps)\varphi\nonumber\\
&\geq&c(c\sigma+2C-1-5\eps)\overline{\varphi}\nonumber\\
&\geq&c(c\sigma-1-5\eps)\overline{\varphi}>0,\quad for~~ c\sigma>1+5\eps,
\end{eqnarray*}
and
\begin{eqnarray*}
(2B&-&2|H|-|E|)\overline{\varphi}-\del_t(B\varphi)-\varphi|D|-\sum_{k=1}^M|\partial_{x_k}(B\varphi)|\\
&\geq&(2\sigma+c^2\sigma-1-c_0\eps)\overline{\varphi}>0,
\end{eqnarray*}
thus it holds
\begin{eqnarray}\label{yw1-2}
\frac{1}{2}\int_{\RR^M}\Big[\pdt^2(A\overline{\varphi})&-&\Delta(B\overline{\varphi})-\pdt(C\overline{\varphi})+\sum_{k,i=1}^M\partial_{x_i}\partial_{x_k}(H_{ki}\overline{\varphi})\nonumber\\
&&\quad+\sum_{k=1}^M\partial_{x_k}\del_t(E_k\overline{\varphi})-\sum_{k=1}^M\partial_{x_k}(D_k\overline{\varphi})-\overline{\varphi}\Big]h^2dx\nonumber\\
&\geq&\frac{1}{2}\int_{\RR^M}\Big(c^2\sigma-c_0(\sigma+1)\eps\Big)\overline{\varphi}h^2dx,
\end{eqnarray}
and
\begin{eqnarray}\label{yw1-3}
\frac{1}{2}\int_{\RR^M}\Big[-\del_t(A\varphi)
&-&\sum_{k=1}^M\partial_{x_k}(E_k\varphi)+(2C-|D|-1)\varphi-(2A+|E|)\overline{\varphi}-\sum_{k=1}^M|\partial_{x_k}(B\varphi)|\Big]h_t^2dx\nonumber\\
&\geq&\frac{1}{2}\int_{\RR^M}c\Big(c\sigma-1-5\eps\Big)\overline{\varphi}h_t^2dx,
\end{eqnarray}
and
\begin{eqnarray}\label{yw1-4}
\frac{1}{2}\int_{\RR^M}\Big[(2B&-&2|H|-|E|)\overline{\varphi}-\del_t(B\varphi)-\varphi|D|-\sum_{k=1}^M|\partial_{x_k}(B\varphi)|\Big]|\nabla h|^2dx\nonumber\\
&\geq&\frac{1}{2}\int_{\RR^M}\Big(2\sigma+c^2\sigma-1-c_0\eps\Big)\overline{\varphi}|\nabla h|^2dx.
\end{eqnarray}

For a given $c_0$, we define
$$
C_{c,\sigma,\eps}:=\min\{c^2\sigma-c_0(\sigma+1)\eps,~c(c\sigma-1-5\eps),~2\sigma+c^2\sigma-1-c_0\eps\},
$$
then, by (\ref{yw1-1})-(\ref{yw1-4}), we can deduce (\ref{TT1}) into
\begin{eqnarray}
\label{yw1-5}
{d\over dt}\int_{\RR^M}\overline{\varphi}\Big((h_t)^2+|\nabla h|^2+h^2
\Big)dx&+&C_{c,\sigma,\eps}\int_{\RR^M}\overline{\varphi}\Big((h_t)^2
+|\nabla h|^2+ h^2\Big)dx\nonumber\\
&\lesssim&\int_{\RR^M}(\varphi+\overline{\varphi})f^2dx.
\end{eqnarray}
Hence we can apply Gronwall's inequality to (\ref{yw1-5}) to obtain
\begin{eqnarray*}
\int_{\RR^M}\overline{\varphi}\Big((h_t)^2
&+&|\nabla h|^2+ h^2\Big)dx\\
&\lesssim& e^{-C_{c,\sigma,\eps}t}\Big[\int_{\RR^M}\Big(h_1^2
+|\nabla h_0|^2+ h_0^2\Big)dx+\int_0^{\infty}\int_{\RR^M}(\varphi+\overline{\varphi})f^2dxdt\Big].
\end{eqnarray*}

\end{proof}

A direct application of lemma 2.1 is to derive the $\LL^2$ estimate of solution for the linear equation (\ref{NYN0-1}).

\begin{lemma}
Let $f\in \CC((0,\infty); \LL^2(\RR^M))$. Assume that (\ref{NYN0-2})-(\ref{NYN0-4}) hold. Then the solution of linear wave equation (\ref{NYN0-1}) satisfies
$$
\int_{\RR^M}\Big((h_t)^2
+|\nabla h|^2+ h^2\Big)dx\lesssim e^{-\eps t}\Big[\int_{\RR^M}\Big(h_0^2
+|\nabla h_0|^2+ h_1^2\Big)dx+\int_0^{\infty}\int_{\RR^M}f^2dxdt\Big].
$$
\end{lemma}
\begin{proof}
For simple, we can take weighted functions
$$
\varphi(t,x):=2e^{-ct},\quad \overline{\varphi}(t,x):=e^{-ct},
$$
with the positive constant $c>1$. One can see that assumptions (\ref{NYN0-5})-(\ref{NYN0-06}) holds.
Then the proof can be shown by following from the proof of lemma 2.1.
\end{proof}

Next, we derive $H^s$-estimates for any $s\geq1$ and $s\in\NN$. In order to keep a similar structure with (\ref{NYN0-1}), we rewrite it as
\bel{-A}
\aligned
h_{tt}&-A^{-1}(t,x)B(t,x)\Delta h+A^{-1}(t,x)C(t,x)h_t+\sum_{k=1}^MA^{-1}(t,x)D_k(t,x)\partial_{x_k}h\\
&+\sum_{k=1}^MA^{-1}(t,x)E_k(t,x)\partial_{x_k}h_t+\sum_{i,k=1}^MA^{-1}(t,x)H_{ki}(t,x)\partial_{x_k}\partial_{x_i}h
=A^{-1}(t,x)f(t,x),~~~~~~~~~
\endaligned
\ee
then we apply $\del_{x_j}^s$ $(\forall j=1,2,\ldots, M)$ to both sides of (\ref{-A}), then we get the linear equation as follows
\bel{E3-14Rr1}
\aligned
\del_{tt}\del_{x_j}^sh&-A^{-1}(t,x)B(t,x)\Delta \del_{x_j}^sh+A^{-1}(t,x)C(t,x)\del_t\del_{x_j}^sh+\sum_{k=1}^MA^{-1}(t,x)D_k(t,x)\partial_{x_k}\del_{x_j}^sh\\
&+\sum_{k=1}^MA^{-1}(t,x)E_k(t,x)\partial_{x_k}\del_t\del_{x_j}^sh+\sum_{i,k=1}^MA^{-1}(t,x)H_{ki}(t,x)\partial_{x_k}\partial_{x_i}\del_{x_j}^sh
=g_s(t,x),~~~~~~~~~
\endaligned
\ee
where
\begin{eqnarray}
\label{E3-14rr0000}
g_s(t,x)&:=&\del_{x_j}^s\Big(A^{-1}f\Big)+\sum\limits_{\substack{s_1+s_2=s\\1\leq s_1\leq s\\0\leq s_2\leq s-1}}
\dbinom{s_2}{s}\del_{x_j}^{s_1}(A^{-1}B)\Delta \del_{x_j}^{s_2}h
-\sum\limits_{\substack{s_1+s_2=s\\1\leq s_1\leq s\\0\leq s_2\leq s-1}}
\dbinom{s_2}{s}\del_{x_j}^{s_1}(A^{-1}C)\del_t\del_{x_j}^{s_2}h\nonumber\\
&&-\sum\limits_{k=1}^M\sum\limits_{\substack{s_1+s_2=s\\1\leq s_1\leq s\\0\leq s_2\leq s-1}}
\dbinom{s_2}{s}\del_{x_j}^{s_1}(A^{-1}D_k)\partial_{x_k}\del_{x_j}^{s_2}h
-\sum\limits_{k=1}^M\sum\limits_{\substack{s_1+s_2=s\\1\leq s_1\leq s\\0\leq s_2\leq s-1}}
\dbinom{s_2}{s}\del_{x_j}^{s_1}(A^{-1}E_k)\partial_{x_k}\del_t\del_{x_j}^{s_2}h\nonumber\\
&&-\sum\limits_{i,k=1}^M\sum\limits_{\substack{s_1+s_2=s\\1\leq s_1\leq s\\0\leq s_2\leq s-1}}
\dbinom{s_2}{s}\del_{x_j}^{s_1}(A^{-1}H_{ki})\partial_{x_k}\partial_{x_i}\del_{x_j}^{s_2}h,
\end{eqnarray}
with the symbol
$$
\dbinom{s_2}{s}={s!\over s_1!s_2!}.
$$

Then we have the following priori estimate.

\begin{lemma}
Let $f\in \CC^1((0,\infty); H^s(\RR^M))$. Then the solution of linear wave equation (\ref{NYN0-1}) satisfies
\begin{eqnarray*}
&&\sum_{j=1}^M\int_{\RR^M}\overline{\varphi}\Big((\del_t\del_{x_j}^sh)^2
+|\nabla\del_{x_j}^s h|^2+ (\del_{x_j}^sh)^2\Big)dx\\
&\lesssim& e^{-C_{c,\sigma,\eps}t}\sum_{j=1}^M\sum_{\theta=0}^s\Big[\int_{\RR^M}\Big((\del_{x_j}^{\theta}h(0,x))^2
+|\nabla\del_{x_j}^{\theta} h(0,x)|^2+ (\del_{x_j}^{\theta}h(0,x))^2\Big)dx\\
&&\quad+\int_0^{\infty}\int_{\RR^M}\Big(\varphi+\overline{\varphi}\Big)(\del_{x_j}^{\theta}f)^2dxdt\Big].
\end{eqnarray*}
\end{lemma}
\begin{proof}
This proof is based on the induction. Let $s=1$ in equation (\ref{E3-14Rr1}). Then
\bel{fg0-1}
\aligned
g_1(t,x)&:=\del_{x_j}\Big(A^{-1}f\Big)+\del_{x_j}(A^{-1}B)\Delta h-\del_{x_j}(A^{-1}C)\del_th-\sum_{k=1}^M\del_{x_j}(A^{-1}D_k)\partial_{x_k}h\\
&\quad-\sum_{k=1}^M\del_{x_j}(A^{-1}E_k)\partial_{x_k}\del_th-\sum_{i,k=1}^M\del_{x_j}(A^{-1}H_{ki})\partial_{x_k}\partial_{x_i}h.
\endaligned
\ee
We notice that linear equation (\ref{E3-14Rr1}) admits the same structure with the linear equation (\ref{NYN0-1}).
So we can multiply both sides of equation (\ref{E3-14Rr1}) with $\varphi\del_{x_j}h$ and $\overline{\varphi}\del_{x_j}\del_th$, respectively, then ultilizing the same process of getting
(\ref{TT1}), we derive
\begin{eqnarray}\label{TT001}
&&\frac{1}{2}\frac{d}{dt}\int_{\RR^M}\Big[(\varphi-\overline{\varphi})(\del_{x_j}\del_th)^2
+\varphi A^{-1}(B-|H|)|\nabla\del_{x_j} h|^2\nonumber\\
&&\quad+\Big(-\overline{\varphi}_t+A^{-1}(C-A)\overline{\varphi}-\sum_{k=1}^M\partial_{x_k}(A^{-1}E_k\overline{\varphi})\Big)(\del_{x_j}h)^2
\Big]dx\nonumber\\
&&\quad+\frac{1}{2}\int_{\RR^M}\Big[\overline{\varphi}_{tt}-\Delta(A^{-1}B\overline{\varphi})-\pdt(A^{-1}C\overline{\varphi})+\sum_{k,i=1}^M\partial_{x_i}\partial_{x_k}(A^{-1}H_{ki}\overline{\varphi})\nonumber\\
&&\quad+\sum_{k=1}^M\partial_{x_k}\del_t(A^{-1}E_k\overline{\varphi})-\sum_{k=1}^M\partial_{x_k}(A^{-1}D_k\overline{\varphi})\Big](\del_{x_j}h)^2dx\nonumber\\
&&\quad+\frac{1}{2}\int_{\RR^M}\Big[-\varphi_t
-\sum_{k=1}^M\partial_{x_k}(A^{-1}E_k\varphi)+A^{-1}(2C-|D|)\varphi-A^{-1}(2A+|E|)\overline{\varphi}\nonumber\\
&&\quad
-\sum_{k=1}^M|\del_{x_k}(A^{-1}B\varphi)|\Big](\del_t\del_{x_j}h)^2dx\nonumber\\
&&\quad+\frac{1}{2}\int_{\RR^M}\Big[A^{-1}(2B-2|H|-|E|)\overline{\varphi}-\del_t(A^{-1}B\varphi)-\varphi A^{-1}|D|-\sum_{k=1}^M|\del_{x_k}(A^{-1}B\varphi)|\Big]|\nabla\del_{x_j} h|^2dx
\nonumber\\
&&\leq \int_{\RR^M}(\varphi\del_{x_j}h_t+\overline{\varphi}\del_{x_j}h)g_1dx.
\end{eqnarray}

We now estimate the right hand side term of (\ref{TT001}). Note that (\ref{fg0-1}). Upon Cauchy inequality, it holds
\begin{eqnarray*}\label{g0f1}
\int_{\RR^M}\varphi\partial_{x_j}(A^{-1}f)\del_t\del_{x_j}hdx
&\leq&\frac{1}{2}\int_{\RR^M}\varphi\Big((\partial_{x_j}(A^{-1}f))^2+(\del_t\del_{x_j}h)^2\Big)dx,\\
\int_{\RR^M}\varphi\partial_{x_j}(A^{-1}B)\Delta h\del_t\del_{x_j}hdx
&\leq&\frac{1}{2}\int_{\RR^M}\varphi|\partial_{x_j}(A^{-1}B)|\Big((\Delta h)^2+(\del_t\del_{x_j}h)^2\Big)dx,\\
-\int_{\RR^M}\varphi\partial_{x_j}(A^{-1}C)h_t\del_t\del_{x_j}hdx
&=&\frac{1}{2}\int_{\RR^M}\partial_{x_j}(\varphi\partial_{x_j}(A^{-1}C))(h_t)^2dx,\\
-\sum_{k=1}^M\int_{\RR^M}\varphi\partial_{x_j}(A^{-1}D_k)\partial_{x_k} h\del_t\del_{x_j}hdx
&\leq&\frac{1}{2}\sum_{k=1}^M\int_{\RR^M}\varphi|\partial_{x_j}(A^{-1}D_k)|\Big((\del_{x_k}h)^2+(\del_t\del_{x_j}h)^2\Big)dx,\\
-\sum_{k=1}^M\int_{\RR^M}\varphi\partial_{x_j}(A^{-1}E_k)\partial_{x_k}h_t\del_t\del_{x_j}hdx
&\leq&{1\over2}\sum_{k=1}^M\int_{\RR^M}\varphi|\partial_{x_j}(A^{-1}E_k)|\Big((\del_t\del_{x_k}h)^2+(\del_t\del_{x_j}h)^2\Big)dx,\\
\int_{\RR^M}\overline{\varphi}\partial_{x_j}(A^{-1}f)\del_{x_j}hdx
&\leq&\frac{1}{2}\int_{\RR^M}\overline{\varphi}\Big((\partial_{x_j}(A^{-1}f))^2+(\del_{x_j}h)^2\Big)dx,\\
\int_{\RR^M}\overline{\varphi}\partial_{x_j}(A^{-1}B)\Delta h\del_{x_j}hdx
&\leq&\frac{1}{2}\int_{\RR^M}\int_{\RR^M}\overline{\varphi}|\partial_{x_j}(A^{-1}B)|\Big((\Delta h)^2+(\del_{x_j}h)^2\Big)dx,\\
-\int_{\RR^M}\overline{\varphi}\partial_{x_j}(A^{-1}C)h_t\del_{x_j}hdx
&\leq&\frac{1}{2}\int_{\RR^M}\overline{\varphi}|\partial_{x_j}( A^{-1}C)|\Big((h_t)^2+(\del_{x_j}h)^2\Big)dx,\\
-\sum_{k=1}^M\int_{\RR^M}\overline{\varphi}\partial_{x_j}(A^{-1}D_k)\partial_{x_k} h\del_{x_j}hdx
&\leq&{1\over2}\sum_{k=1}^M\int_{\RR^M}\overline{\varphi}|\partial_{x_j}(A^{-1}D_k)|\Big((\del_{x_k}h)^2+(\del_{x_j}h)^2\Big)dx,
\end{eqnarray*}
and
\begin{eqnarray*}
&&-\sum_{k=1}^M\int_{\RR^M}\overline{\varphi}\partial_{x_j}(A^{-1}E_k)\partial_{x_k}h_t\del_{x_j}hdx\\
&\leq&\frac{1}{2}\sum_{k=1}^M\int_{\RR^M}\overline{\varphi}|\partial_{x_j}(A^{-1}E_k)|\Big( (\del_t\del_{x_k}h)^2+(\del_{x_j}h)^2\Big)dx,\\
&&-\sum_{k,i=1}^M\int_{\RR^M}\overline{\varphi}\partial_{x_j}(A^{-1}H_{ki})\partial_{x_k}\partial_{x_i}h\del_{x_j}hdx\\
&\leq&\frac{1}{2}\sum_{k,i=1}^M\int_{\RR^M}\overline{\varphi}|\partial_{x_j}(A^{-1}H_{ki})|\Big((\del_{x_k}\del_{x_i}h)^2+(\del_{x_j}h)^2\Big)dx,\\
&&-\sum_{k,i=1}^M\int_{\RR^M}\varphi\partial_{x_j}(A^{-1}H_{ki})\partial_{x_k}\partial_{x_i}h\del_t\del_{x_j}hdx\\
&\leq&\frac{1}{2}\sum_{k,i=1}^M\int_{\RR^M}\varphi|\partial_{x_j}(A^{-1}H_{ki})|\Big((\del_{x_k}\del_{x_i}h)^2+(\del_t\del_{x_j}h)^2\Big)dx,
\end{eqnarray*}
thus, based on above estimates,  we sum up (\ref{TT001}) from $j=1$ to $j=M$, then it reduces into
\begin{eqnarray}\label{0TT002}
&&\frac{1}{2}\frac{d}{dt}\sum_{j=1}^M\int_{\RR^M}\Big[(\varphi-\overline{\varphi})(\del_{x_j}\del_th)^2
+\varphi A^{-1}(B-|H|)|\nabla\del_{x_j} h|^2\nonumber\\
&&\quad+\Big(-\overline{\varphi}_t+A^{-1}(C-A)\overline{\varphi}-\sum_{k=1}^M\partial_{x_k}(A^{-1}E_k\overline{\varphi})\Big)(\del_{x_j}h)^2
\Big]dx
+\frac{1}{2}\sum_{j=1}^M\int_{\RR^M}\overline{A}_1(t,x)(\del_{x_j}h)^2dx\nonumber\\
&&\quad
+\frac{1}{2}\sum_{j=1}^M\int_{\RR^M}\overline{A}_2(t,x)(\del_t\del_{x_j}h)^2dx
+\frac{1}{2}\sum_{j=1}^M\int_{\RR^M}\overline{A}_3(t,x)|\nabla\del_{x_j} h|^2dx
\nonumber\\
&&\leq \frac{1}{2}\sum_{j=1}^M\int_{\RR^M}(\varphi+\overline{\varphi})(\del_{x_j}(A^{-1}f))dx
+\frac{1}{2}\int_{\RR^M}(\varphi+\overline{\varphi})\sum_{j,k=1}^M|\partial_{x_j}(A^{-1}D_k)||\nabla h|^2dx\nonumber\\
&&\quad+\frac{1}{2}\sum_{j=1}^M\int_{\RR^M}\Big[\partial_{x_j}(\varphi\partial_{x_j}(A^{-1}C))+\overline{\varphi}|\partial_{x_j}( A^{-1}C)|\Big](h_t)^2dx,~~~~~~~
\end{eqnarray}
where
\begin{eqnarray*}
\overline{A}_1(t,x)&:=&\overline{\varphi}_{tt}-\Delta(A^{-1}B\overline{\varphi})-\pdt(A^{-1}C\overline{\varphi})
+\sum_{k,i=1}^M\partial_{x_i}\partial_{x_k}(A^{-1}H_{ki}\overline{\varphi})+\sum_{k=1}^M\partial_{x_k}\del_t(A^{-1}E_k\overline{\varphi})\nonumber\\
&&\quad-\sum_{k=1}^M\partial_{x_k}(A^{-1}D_k\overline{\varphi})
-\overline{\varphi}\Big(1+|\partial_{x_j}(A^{-1}B)|+|\partial_{x_j}( A^{-1}C)|\nonumber\\
&&\quad+\sum_{k=1}^M|\partial_{x_j}(A^{-1}D_k)|+\sum_{k=1}^M|\partial_{x_j}(A^{-1}E_k)|+\sum_{k,i=1}^M|\partial_{x_j}(A^{-1}H_{ki})|\Big),\\
\overline{A}_2(t,x)&:=&-\varphi_t-\sum_{k=1}^M\partial_{x_k}(A^{-1}E_k\varphi)
+\varphi A^{-1}(2C-|D|)-A^{-1}(2A+|E|)\overline{\varphi}
-\sum_{k=1}^M|\del_{x_k}(A^{-1}B\varphi)|\nonumber\\
&&\quad
-(\varphi+\overline{\varphi})\sum_{k=1}^M|\partial_{x_k}(A^{-1}E_j)|
+\varphi\Big(-1-|\partial_{x_j}(A^{-1}B)|-\sum_{k=1}^M|\partial_{x_j}(A^{-1}D_k)|\nonumber\\
&&\quad
-\sum_{k=1}^M|\partial_{x_j}(A^{-1}E_k)|-\sum_{k,i=1}^M|\partial_{x_j}(A^{-1}H_{ki})|\Big),\\
\overline{A}_3(t,x)&:=&A^{-1}(2B-2|H|-|E|)\overline{\varphi}-\del_t(A^{-1}B\varphi)
-\varphi A^{-1}|D|-\sum_{k=1}^M|\del_{x_k}(A^{-1}B\varphi)|\\
&&\quad-(\varphi+\overline{\varphi})\Big(|\partial_{x_j}(A^{-1}B)|+\sum_{k,i=1}^M|\partial_{x_i}(A^{-1}H_{kj})|\Big).
\end{eqnarray*}

We now analyze all of coefficients for inequality (\ref{0TT002}).
By the assumption given in (\ref{NYN0-2})-(\ref{NYN0-3}) and (\ref{NYN0-6}), it holds
$$
\aligned
&\varphi-\overline{\varphi}>c\overline{\varphi} ,\\
&\varphi A^{-1}(B-|H|)>\varphi(\sigma-{1\over2}),
\endaligned
$$
and by (\ref{NYN0-04}) and (\ref{NYN0-7}), we have
$$
\aligned
&\quad-\overline{\varphi}_t+A^{-1}(C-A)\overline{\varphi}-\sum_{k=1}^M\partial_{x_k}(A^{-1}E_k\overline{\varphi})\\
&=
A^{-1}\Big(-\pdt(A\overline{\varphi})+(C-A)\overline{\varphi}-\sum_{k=1}^M\partial_{x_k}(E_k\overline{\varphi})\Big)
+(A^{-1}A_t\overline{\varphi}+\sum_{k=1}^M\del_{x_k}A^{-1}E_k\overline{\varphi})\\
&\geq A^{-1}\Big(\Big(A(c-1)-A_t+C\Big)\overline{\varphi}-\sum_{k=1}^M\partial_{x_k}(E_k\overline{\varphi})\Big)+o(\eps)\\
&\geq A^{-1}\Big(\sigma(c-1)-3\eps\Big)\overline{\varphi}+o(\eps)>0,
\endaligned
$$
thus it holds
\begin{eqnarray}
\label{c-t}
&&\int_{\RR^M}\Big[(\varphi-\overline{\varphi})(\del_t\del_{x_j}h)^2
+\varphi A^{-1}(B-|H|)|\nabla\del_{x_j}h|^2\nonumber\\
&&\quad
+\Big(-\overline{\varphi}_t
+A^{-1}(C-A)\overline{\varphi}-\sum_{k=1}^M\partial_{x_k}(A^{-1}E_k\overline{\varphi})\Big)(\del_{x_j}h)^2
\Big]dx\nonumber\\
&&\gtrsim\int_{\RR^M}\Big[c\overline{\varphi}(\del_t\del_{x_j}h)^2
+\varphi(\sigma-{1\over2})|\nabla\del_{x_j}h|^2
+\Big(\overline{\varphi}\Big(\sigma(c-1)-3\eps\Big)+o(\eps)\Big)(\del_{x_j}h)^2
\Big]dx.~~~~~~
\end{eqnarray}

Similarly, using (\ref{NYN0-2})-(\ref{NYN0-6}), there exists a positive constant $c_0$ such that
\begin{eqnarray*}
\overline{A}_1(t,x)&:=&
A^{-1}\Big(\pdt^2(A\overline{\varphi})-\Delta(B\overline{\varphi})
-\pdt(C\overline{\varphi})
+\sum_{k,i=1}^M\partial_{x_i}\partial_{x_k}(H_{ki}\overline{\varphi})
+\sum_{k=1}^M\del_t\partial_{x_k}(E_k\overline{\varphi})\nonumber\\
&&\quad-\sum_{k=1}^M\partial_{x_k}(D_k\overline{\varphi})-\overline{\varphi}\Big)-A^{-1}\del_{t}^2A\overline{\varphi}-A^{-1}A_t\overline{\varphi}_t
-\sum_{k=1}^M\del_{x_k}^2( A^{-1})B\overline{\varphi}\nonumber\\
&&\quad-\sum_{k=1}^M\del_{x_k} A^{-1}\del_{x_k}( B\overline{\varphi})-A^{-1}_tC\overline{\varphi}
+\sum_{k,i=1}^M\partial_{x_i}\partial_{x_k}A^{-1}H_{ki}\overline{\varphi}
+\sum_{k,i=1}^M\partial_{x_i}A^{-1}\partial_{x_k}(H_{ki}\overline{\varphi})\\
&&\quad+\sum_{k,i=1}^M\partial_{x_k}A^{-1}\partial_{x_i}(H_{ki}\overline{\varphi})
+\sum_{k=1}^M\del_t\partial_{x_k}A^{-1}E_k\overline{\varphi}
+\sum_{k=1}^M\partial_{x_k}A^{-1}\del_t(E_k\overline{\varphi})\\
&&\quad+\sum_{k=1}^M\del_tA^{-1}\partial_{x_k}(E_k\overline{\varphi})
-\sum_{k=1}^M\partial_{x_k}A^{-1}D_k\overline{\varphi}
+A^{-1}\overline{\varphi}
-\varphi\sum_{k=1}^M|\partial_{x_j}(A^{-1}D_k)|\nonumber\\
&&\quad
-\overline{\varphi}\Big(1+|\partial_{x_j}(A^{-1}B)|
+|\partial_{x_j}( A^{-1}C)|
+\sum_{k=1}^M|\partial_{x_j}(A^{-1}D_k)|\\
&&\quad+\sum_{k=1}^M|\partial_{x_j}(A^{-1}E_k)|+\sum_{k,i=1}^M|\partial_{x_j}(A^{-1}H_{ki})|\Big)\\
&\geq&A^{-1}\Big((c^2-1)\sigma-c_0(\sigma+1)\eps+1\Big)\overline{\varphi}+o(\eps)>0,
\end{eqnarray*}
and
\begin{eqnarray*}
&&\overline{A}_2(t,x)\\
&:=&A^{-1}\Big(-\del_t(A\varphi)
-\sum_{k=1}^M\partial_{x_k}(E_k\varphi)+(2C-|D|-1)\varphi-(2A+|E|)\overline{\varphi}
-\sum_{k=1}^M|\del_{x_k}(B\varphi)|
\Big)\nonumber\\
&&+A^{-1}A_t\varphi
-\sum_{k=1}^M\partial_{x_k}A^{-1}E_k\varphi
+A^{-1}\varphi-\sum_{k=1}^M|\del_{x_k}A^{-1}B\varphi|
-(\varphi+\overline{\varphi})\sum_{k=1}^M|\partial_{x_k}(A^{-1}E_j)|\\
&&
+\varphi\Big(-1-|\partial_{x_j}(A^{-1}B)|
-\sum_{k=1}^M|\partial_{x_j}(A^{-1}D_k)|-\sum_{k=1}^M|\partial_{x_j}(A^{-1}E_k)|
-\sum_{k,i=1}^M|\partial_{x_j}(A^{-1}H_{ki})|\Big)\\
&\geq&A^{-1}\Big((c^2-1)\sigma-c_0(\sigma+1)\eps\Big)\overline{\varphi}+o(\eps)>0,\\
\end{eqnarray*}
and
\begin{eqnarray*}
\overline{A}_3(t,x)&:=&A^{-1}\Big((2B-2|H|-|E|)\overline{\varphi}-\del_t(B\varphi)-\varphi|D|-\sum_{k=1}^M|\del_{x_k}(B\varphi)|\Big)
-A^{-1}_tB\varphi\\
&&-\sum_{k=1}^M|\del_{x_k}A^{-1}B\varphi|-\Big(\varphi+\overline{\varphi}\Big)\Big(|\partial_{x_j}(A^{-1}B)|+\sum_{k,i=1}^M|\partial_{x_j}(A^{-1}H_{ki})|\Big)\\
&\geq&A^{-1}(2\sigma+c^2\sigma-1-c_0\eps)\overline{\varphi}+o(\eps)>0,
\end{eqnarray*}
thus it holds
\begin{eqnarray*}\label{c1-A1}
\frac{1}{2}\int_{\RR^M}\overline{A}_1(t,x)(\del_{x_j}h)^2dx
\geq\int_{\RR^M}\Big[\Big(c^2\sigma-c_0(\sigma+1)\eps\Big)\overline{\varphi}+o(\eps)\Big](\del_{x_j}h)^2dx,\\
\frac{1}{2}\int_{\RR^M}\overline{A}_2(t,x)(\del_t\del_{x_j}^sh)^2dx
\geq\frac{1}{2}\int_{\RR^M}\Big[c(c\sigma-1-5\eps)\overline{\varphi}+o(\eps)\Big](\del_t\del_{x_j}h)^2dx,\\
\frac{1}{2}\int_{\RR^M}\overline{A}_3(t,x)|\nabla\del_{x_j}^s h|^2dx
\geq\frac{1}{2}\int_{\RR^M}\Big[\Big(2\sigma+c^2\sigma-1-c_0\eps\Big)\overline{\varphi}+o(\eps)\Big]|\nabla\del_{x_j} h|^2dx.
\end{eqnarray*}

Upon above estimates, there exists a positive constants $C_{c,\sigma,\eps}$ depending on $c$, $\sigma$ and $\eps$ such that
\begin{eqnarray*}
\overline{A}_1(t,x)\geq C_{c,\sigma,\eps},~~\overline{A}_2(t,x)\geq C_{c,\sigma,\eps},~~\overline{A}_3(t,x)\geq C_{c,\sigma,\eps},
\end{eqnarray*}
thus (\ref{0TT002}) is reduced into
\begin{eqnarray*}\label{TT002}
&&\int_{\RR^M}\overline{\varphi}\Big(\del_{x_j}\del_th)^2+|\nabla\del_{x_j} h|^2+(\del_{x_j}h)^2\Big)dx\\
&&\quad+C_{c,\sigma,\eps}\int_0^t\int_{\RR^M}\overline{\varphi}\Big(\del_{x_j}\del_th)^2+|\nabla\del_{x_j} h|^2+(\del_{x_j}h)^2\Big)dxds\nonumber\\
&&\lesssim\int_0^t\int_{\RR^M}(\varphi+\overline{\varphi})\Big(f^2+(\del_{x_j}f)^2+(h_t)^2\Big)dxds,~~~~~~~
\end{eqnarray*}
from which, upon Grownwall's inequality, it holds
\begin{eqnarray*}
\label{wyw01-1}
&&\sum_{j=1}^M\int_{\RR^M}\overline{\varphi}\Big((\del_t\del_{x_j}h)^2
+|\nabla\del_{x_j} h|^2+(\del_{x_j} h)^2\Big)dx\nonumber\\
&\lesssim& e^{-C_{c,\sigma,\eps}t}\sum_{\theta=0}^1\sum_{j=1}^M\Big[\int_{\RR^M}\overline{\varphi}(0,x)\Big((\del^{\theta}_{x_j}\del_th(0,x))^2
+|\nabla \del_{x_j}^{\theta}h(0,x)|^2+(\del_{x_j}^{\theta}h(0,x))^2\Big)dx\nonumber\\
&&+\int_0^{\infty}\int_{\RR^M}\overline{\varphi}\Big(f^2+(\del_{x_j}f)^2\Big)dxdt\Big],
\end{eqnarray*}
where we use the result of lemma 2.1 to estimate the term of $(h_t)^2$ and $|\nabla h|^2$.

Let $1\leq\theta\leq s-1$. We assume that
\begin{eqnarray}
\label{wyw01-1}
&&\sum_{j=1}^M\int_{\RR^M}\overline{\varphi}\Big((\del_t\del_{x_j}^{\theta}h)^2
+|\nabla\del_{x_j}^{\theta} h|^2+(\del_{x_j}^{\theta}h)^2\Big)dx\nonumber\\
&\lesssim& e^{-C_{c,\sigma,\eps}t}\sum_{\theta=0}^{s-1}\sum_{j=1}^M\Big[\int_{\RR^M}\overline{\varphi}(0,x)\Big((\del_{x_j}^{\theta}\del_th(0,x))^2
+|\nabla \del_{x_j}^{\theta}h(0,x)|^2+(\del_{x_j}^{\theta}h(0,x))^2\Big)dx\nonumber\\
&&+\int_0^{\infty}\int_{\RR^M}\overline{\varphi}(\del_{x_j}^{\theta}f)^2dxdt\Big]
\end{eqnarray}
holds. Then we prove the case $\theta=s$ also holds.

We multiply both sides of (\ref{E3-14Rr1}) with $\varphi(t,x)\del_t\del_{x_j}^s h$ and $\overline{\varphi}(t,x)\del_{x_j}^s h$, respectively, then similar to get \eqref{TT001}, it holds
\begin{eqnarray}\label{TT001s}
&&\frac{1}{2}\frac{d}{dt}\int_{\RR^M}\Big[(\varphi-\overline{\varphi})(\del_{x_j}^s\del_th)^2
+\varphi A^{-1}(B-|H|)|\nabla\del_{x_j}^s h|^2\nonumber\\
&&\quad+\Big(-\overline{\varphi}_t+A^{-1}(C-A)\overline{\varphi}-\sum_{k=1}^M\partial_{x_k}(A^{-1}E_k\overline{\varphi})\Big)(\del_{x_j}^sh)^2
\Big]dx\nonumber\\
&&\quad+\frac{1}{2}\int_{\RR^M}\Big[\overline{\varphi}_{tt}-\Delta(A^{-1}B\overline{\varphi})-\pdt(A^{-1}C\overline{\varphi})+\sum_{k,i=1}^M\partial_{x_i}\partial_{x_k}(A^{-1}H_{ki}\overline{\varphi})\nonumber\\
&&\quad+\sum_{k=1}^M\partial_{x_k}\del_t(A^{-1}E_k\overline{\varphi})-\sum_{k=1}^M\partial_{x_k}(A^{-1}D_k\overline{\varphi})\Big](\del_{x_j}^sh)^2dx\nonumber\\
&&\quad+\frac{1}{2}\int_{\RR^M}\Big[-\varphi_t
-\sum_{k=1}^M\partial_{x_k}(A^{-1}E_k\varphi)+A^{-1}(2C-|D|)\varphi-A^{-1}(2A+|E|)\overline{\varphi}\nonumber\\
&&\quad
-\sum_{k=1}^M|\del_{x_k}(A^{-1}B\varphi)|\Big](\del_t\del_{x_j}^sh)^2dx\nonumber\\
&&\quad+\frac{1}{2}\int_{\RR^M}\Big[A^{-1}(2B-2|H|-|E|)\overline{\varphi}-\del_t(A^{-1}B\varphi)-\varphi A^{-1}|D|-\sum_{k=1}^M|\del_{x_k}(A^{-1}B\varphi)|\Big]|\nabla\del_{x_j}^s h|^2dx
\nonumber\\
&&\leq \int_{\RR^M}(\varphi\del_{x_j}^sh_t+\overline{\varphi}\del_{x_j}^sh)g_sdx.
\end{eqnarray}

Firstly, we estimate the right hand side of (\ref{TT001s}). Upon (\ref{E3-14rr0000}) and Cauchy's inequality, it holds
\begin{eqnarray*}\label{gsf1}
\int_{\RR^M}\varphi\del_{x_j}^s\Big(A^{-1}f\Big)\del_t\del_{x_j}^shdx
\leq\frac{1}{2}\int_{\RR^M}\varphi\Big((\del_{x_j}^s(A^{-1}f))^2+(\del_t\del_{x_j}^sh)^2\Big)dx,
\end{eqnarray*}
and
\begin{eqnarray*}
&&\sum\limits_{\substack{s_1+s_2=s\\1\leq s_1\leq s\\0\leq s_2\leq s-1}}
\int_{\RR^M}\dbinom{s_2}{s}\varphi\del_{x_j}^{s_1}(A^{-1}B)\Delta \del_{x_j}^{s_2}h\del_t\del_{x_j}^shdx\\
&&\quad\lesssim
\sum\limits_{\substack{s_1+s_2=s\\1\leq s_1\leq s\\0\leq s_2\leq s-1}}
\int_{\RR^M}\varphi|\del_{x_j}^{s_1}(A^{-1}B)|\Big((\Delta \del_{x_j}^{s_2}h)^2+(\del_t\del_{x_j}^sh)^2\Big)dx,\nonumber\\
&&\sum\limits_{\substack{s_1+s_2=s\\1\leq s_1\leq s\\0\leq s_2\leq s-1}}
\int_{\RR^M}\dbinom{s_2}{s}\varphi\del_{x_j}^{s_1}(A^{-1}C)\del_t\del_{x_j}^{s_2}h\del_t\del_{x_j}^shdx\\
&&\quad\lesssim
\sum\limits_{\substack{s_1+s_2=s\\1\leq s_1\leq s\\0\leq s_2\leq s-1}}
\int_{\RR^M}\varphi|\del_{x_j}^{s_1}(A^{-1}C)|\Big((\del_t\del_{x_j}^{s_2}h)^2+(\del_t\del_{x_j}^sh)^2\Big)dx,\\
&&\sum_{k=1}^M\sum\limits_{\substack{s_1+s_2=s\\1\leq s_1\leq s\\0\leq s_2\leq s-1}}
\int_{\RR^M}\dbinom{s_2}{s}\varphi\del_{x_j}^{s_1}(A^{-1}D_k)\partial_{x_k}\del_{x_j}^{s_2}h\del_t\del_{x_j}^shdx\\
&&\quad\lesssim
\sum_{k=1}^M\sum\limits_{\substack{s_1+s_2=s\\1\leq s_1\leq s\\0\leq s_2\leq s-1}}
\int_{\RR^M}\varphi|\del_{x_j}^{s_1}(A^{-1}D_k)|\Big((\partial_{x_k}\del_{x_j}^{s_2}h)^2+(\del_t\del_{x_j}^sh)^2\Big)dx,\nonumber\\
&&\sum_{k=1}^M\sum\limits_{\substack{s_1+s_2=s\\1\leq s_1\leq s\\0\leq s_2\leq s-1}}
\int_{\RR^M}\dbinom{s_2}{s}\varphi\del_{x_j}^{s_1}(A^{-1}E_k)\partial_{x_k}\del_t\del_{x_j}^{s_2}h\del_t\del_{x_j}^shdx\\
&&\quad\lesssim
\sum_{k=1}^M\sum\limits_{\substack{s_1+s_2=s\\1\leq s_1\leq s\\0\leq s_2\leq s-1}}
\int_{\RR^M}\varphi|\del_{x_j}^{s_1}(A^{-1}E_k)|\Big((\partial_{x_k}\del_t\del_{x_j}^{s_2}h)^2+(\del_t\del_{x_j}^sh)^2\Big)dx,\nonumber\\
&&\sum_{k,i=1}^M\sum\limits_{\substack{s_1+s_2=s\\1\leq s_1\leq s\\0\leq s_2\leq s-1}}
\int_{\RR^M}\dbinom{s_2}{s}\varphi\del_{x_j}^{s_1}(A^{-1}H_{ki})\partial_{x_k}\partial_{x_i}\del_{x_j}^{s_2}h\del_t\del_{x_j}^shdx\\
&&\quad\lesssim
\sum_{k,i=1}^M\sum\limits_{\substack{s_1+s_2=s\\1\leq s_1\leq s\\0\leq s_2\leq s-1}}
\int_{\RR^M}\varphi|\del_{x_j}^{s_1}(A^{-1}H_{ki})|\Big((\partial_{x_k}\partial_{x_i}\del_{x_j}^{s_2}h)^2+(\del_t\del_{x_j}^sh)^2\Big)dx,
\end{eqnarray*}
and
\begin{eqnarray*}
\int_{\RR^M}\overline{\varphi}\del_{x_j}^s\Big(A^{-1}f\Big)\del_{x_j}^shdx
\leq\frac{1}{2}\int_{\RR^M}\overline{\varphi}\Big((\del_{x_j}^s(A^{-1}f))^2+(\del_{x_j}^sh)^2\Big),
\end{eqnarray*}
\begin{eqnarray*}
&&\sum\limits_{\substack{s_1+s_2=s\\1\leq s_1\leq s\\0\leq s_2\leq s-1}}
\int_{\RR^M}\dbinom{s_2}{s}\overline{\varphi}\del_{x_j}^{s_1}(A^{-1}B)\Delta \del_{x_j}^{s_2}h\del_{x_j}^shdx\\
&&\quad\lesssim
\sum\limits_{\substack{s_1+s_2=s\\1\leq s_1\leq s\\0\leq s_2\leq s-1}}
\int_{\RR^M}\overline{\varphi}|\del_{x_j}^{s_1}(A^{-1}B)|\Big((\Delta \del_{x_j}^{s_2}h)^2+(\del_{x_j}^sh)^2\Big)dx,
\nonumber\\
&&\sum\limits_{\substack{s_1+s_2=s\\1\leq s_1\leq s\\0\leq s_2\leq s-1}}
\int_{\RR^M}\dbinom{s_2}{s}\overline{\varphi}\del_{x_j}^{s_1}(A^{-1}C)\del_t\del_{x_j}^{s_2}h\del_{x_j}^shdx\\
&&\quad\lesssim
\sum\limits_{\substack{s_1+s_2=s\\1\leq s_1\leq s\\0\leq s_2\leq s-1}}
\int_{\RR^M}\overline{\varphi}|\del_{x_j}^{s_1}(A^{-1}C)|\Big((\del_t\del_{x_j}^{s_2}h)^2+(\del_{x_j}^sh)^2\Big)dx,
\nonumber\\
&&\sum_{k=1}^M\sum\limits_{\substack{s_1+s_2=s\\1\leq s_1\leq s\\0\leq s_2\leq s-1}}
\int_{\RR^M}\dbinom{s_2}{s}\overline{\varphi}\del_{x_j}^{s_1}(A^{-1}D_k)\partial_{x_k}\del_{x_j}^{s_2}h\del_{x_j}^shdx\\
&&\quad\lesssim
\sum_{k=1}^M\sum\limits_{\substack{s_1+s_2=s\\1\leq s_1\leq s\\0\leq s_2\leq s-1}}
\int_{\RR^M}\overline{\varphi}|\del_{x_j}^{s_1}(A^{-1}D_k)|\Big((\partial_{x_k}\del_{x_j}^{s_2}h)^2+(\del_{x_j}^sh)^2\Big)dx,
\nonumber\\
&&\sum_{k=1}^M\sum\limits_{\substack{s_1+s_2=s\\1\leq s_1\leq s\\0\leq s_2\leq s-1}}
\int_{\RR^M}\dbinom{s_2}{s}\overline{\varphi}\del_{x_j}^{s_1}(A^{-1}E_k)\partial_{x_k}\del_t\del_{x_j}^{s_2}h\del_{x_j}^shdx\\
&&\quad\lesssim
\sum_{k=1}^M\sum\limits_{\substack{s_1+s_2=s\\1\leq s_1\leq s\\0\leq s_2\leq s-1}}
\int_{\RR^M}\overline{\varphi}|\del_{x_j}^{s_1}(A^{-1}E_k)|\Big((\del_{x_k}\del_t\del_{x_j}^{s_2}h)^2+(\del_{x_j}^sh)^2\Big)dx,
\nonumber\\
&&\sum_{k,i=1}^M\sum\limits_{\substack{s_1+s_2=s\\1\leq s_1\leq s\\0\leq s_2\leq s-1}}
\int_{\RR^M}\dbinom{s_2}{s}\overline{\varphi}\del_{x_j}^{s_1}(A^{-1}H_{ki})\partial_{x_k}\partial_{x_i}\del_{x_j}^{s_2}h\del_{x_j}^shdx\\
&&\quad\lesssim
\sum_{k,i=1}^M\sum\limits_{\substack{s_1+s_2=s\\1\leq s_1\leq s\\0\leq s_2\leq s-1}}
\int_{\RR^M}\overline{\varphi}|\del_{x_j}^{s_1}(A^{-1}H_{ki})|\Big((\partial_{x_k}\partial_{x_i}\del_{x_j}^{s_2}h)^2+(\del_{x_j}^sh)^2\Big)dx,
\end{eqnarray*}
upon above estimates,  we sum up (\ref{TT001s}) from $j=1$ to $j=M$, then we can reduce it into
\begin{eqnarray}\label{TT001sc1}
&&\frac{1}{2}\frac{d}{dt}\sum_{j=1}^M\int_{\RR^M}\Big[
\Big(-\overline{\varphi}_t+A^{-1}(C-A)\overline{\varphi}-\sum_{k=1}^M\partial_{x_k}(A^{-1}E_k\overline{\varphi})\Big)(\del_{x_j}^sh)^2
+\varphi A^{-1}(B-|H|)|\nabla\del_{x_j}^s h|^2\nonumber\\
&&\quad
+(\varphi-\overline{\varphi})(\del_{x_j}^s\del_th)^2
\Big]dx
+\frac{1}{2}\sum_{j=1}^M\int_{\RR^M}\widetilde{A}_1(t,x)(\del_{x_j}^sh)^2dx
+\frac{1}{2}\sum_{j=1}^M\int_{\RR^M}\widetilde{A}_2(t,x)(\del_t\del_{x_j}^sh)^2dx\nonumber\\
&&\quad+\frac{1}{2}\sum_{j=1}^M\int_{\RR^M}\widetilde{A}_3(t,x)|\nabla\del_{x_j}^s h|^2dx
\leq\frac{1}{2}\sum_{j=1}^M\int_{\RR^M}\Big(\varphi+\overline{\varphi}\Big)(\del_{x_j}^s(A^{-1}f))^2dx+\sum_{j=1}^MR_j,
\end{eqnarray}
where
\begin{eqnarray*}
\widetilde{A}_1(t,x)&:=&
\overline{\varphi}_{tt}-\Delta(A^{-1}B\overline{\varphi})-\pdt(A^{-1}C\overline{\varphi})+\sum_{k,i=1}^M\partial_{x_i}\partial_{x_k}(A^{-1}H_{ki}\overline{\varphi})
+\sum_{k=1}^M\partial_{x_k}\del_t(A^{-1}E_k\overline{\varphi})\\
&&-\sum_{k=1}^M\partial_{x_k}(A^{-1}D_k\overline{\varphi})
-2\overline{\varphi}\Big({1\over2}
+\sum\limits_{\substack{s_1+s_2=s\\1\leq s_1\leq s\\0\leq s_2\leq s-1}}|\del_{x_j}^{s_1}(A^{-1}B)|
+\sum\limits_{\substack{s_1+s_2=s\\1\leq s_1\leq s\\0\leq s_2\leq s-1}}|\del_{x_j}^{s_1}(A^{-1}C)|\\
&&+\sum_{k=1}^M\sum\limits_{\substack{s_1+s_2=s\\1\leq s_1\leq s\\0\leq s_2\leq s-1}}|\del_{x_j}^{s_1}(A^{-1}D_k)|
+\sum_{k=1}^M\sum\limits_{\substack{s_1+s_2=s\\1\leq s_1\leq s\\0\leq s_2\leq s-1}}|\del_{x_j}^{s_1}(A^{-1}E_k)|\\
&&\quad+\sum_{k,i=1}^M\sum\limits_{\substack{s_1+s_2=s\\1\leq s_1\leq s\\0\leq s_2\leq s-1}}|\del_{x_j}^{s_1}(A^{-1}H_{ki})|\Big),\\
\widetilde{A}_2(t,x)&:=&
-\varphi_t-\sum_{k=1}^M\partial_{x_k}(A^{-1}E_k\varphi)+A^{-1}(2C-|D|)\varphi-A^{-1}(2A+|E|)\overline{\varphi}
-\sum_{k=1}^M|\del_{x_k}(A^{-1}B\varphi)|\nonumber\\
&&
-2\varphi\Big({1\over2}+\sum\limits_{\substack{s_1+s_2=s\\1\leq s_1\leq s\\0\leq s_2\leq s-1}}|\del_{x_j}^{s_1}(A^{-1}B)|
+\sum\limits_{\substack{s_1+s_2=s\\1\leq s_1\leq s\\0\leq s_2\leq s-1}}|\del_{x_j}^{s_1}(A^{-1}C)|
+\sum_{k=1}^M\sum\limits_{\substack{s_1+s_2=s\\1\leq s_1\leq s\\0\leq s_2\leq s-1}}|\del_{x_j}^{s_1}(A^{-1}D_k)|\nonumber\\
&&+\sum_{k=1}^M\sum\limits_{\substack{s_1+s_2=s\\1\leq s_1\leq s\\0\leq s_2\leq s-1}}|\del_{x_j}^{s_1}(A^{-1}E_k)|
+\sum_{k,i=1}^M\sum\limits_{\substack{s_1+s_2=s\\1\leq s_1\leq s\\0\leq s_2\leq s-1}}|\del_{x_j}^{s_1}(A^{-1}H_{ki})|\Big),\\
\widetilde{A}_3(t,x)&:=&
A^{-1}(2B-2|H|-|E|)\overline{\varphi}-\del_t(A^{-1}B\varphi)-\varphi A^{-1}|D|-\sum_{k=1}^M|\del_{x_k}(A^{-1}B\varphi)|,\\
R_j(t,x)&:=&\sum\limits_{\substack{s_1+s_2=s\\1\leq s_1\leq s\\0\leq s_2\leq s-1}}
\int_{\RR^M}\Big(\varphi+\overline{\varphi}\Big)|\del_{x_j}^{s_1}(A^{-1}B)|(\Delta\del_{x_j}^{s_2}h)^2dx\nonumber\\
&&+
\sum\limits_{\substack{s_1+s_2=s\\1\leq s_1\leq s\\0\leq s_2\leq s-1}}
\int_{\RR^M}\Big(\varphi+\overline{\varphi}\Big)|\del_{x_j}^{s_1}(A^{-1}C)|(\del_t\del_{x_j}^{s_2}h)^2dx\nonumber\\
&&+
\sum_{k=1}^M\sum\limits_{\substack{s_1+s_2=s\\1\leq s_1\leq s\\0\leq s_2\leq s-1}}
\int_{\RR^M}\Big(\varphi+\overline{\varphi}\Big)|\del_{x_j}^{s_1}(A^{-1}D_k)|(\del_{x_k}\del_{x_j}^{s_2}h)^2dx\nonumber\\
&&+
\sum_{k=1}^M\sum\limits_{\substack{s_1+s_2=s\\1\leq s_1\leq s\\0\leq s_2\leq s-1}}
\int_{\RR^M}\Big(\varphi+\overline{\varphi}\Big)|\del_{x_j}^{s_1}(A^{-1}E_k)|(\del_{x_k}\del_t\del_{x_j}^{s_2}h)^2dx\nonumber\\
&&+
\sum_{k,i=1}^M\sum\limits_{\substack{s_1+s_2=s\\1\leq s_1\leq s\\0\leq s_2\leq s-1}}
\int_{\RR^M}\Big(\varphi+\overline{\varphi}\Big)|\del_{x_j}^{s_1}(A^{-1}H_{ki})|(\partial_{x_k}\partial_{x_i}\del_{x_j}^{s_2}h)^2dx.
\end{eqnarray*}

Furthermore, by (\ref{NYN0-04}) and (\ref{NYN0-7}), it holds
\begin{eqnarray*}
&&A^{-1}\Big(-\pdt(A\overline{\varphi})+(C-A)\overline{\varphi}-\sum_{k=1}^M\partial_{x_k}(E_k\overline{\varphi})\Big)
+(A^{-1}A_t\overline{\varphi}+\sum_{k=1}^M\del_{x_k}A^{-1}E_k\overline{\varphi})\nonumber\\
&\geq&A^{-1}\Big(\Big(A(c-1)-A_t+C\Big)\overline{\varphi}-\sum_{k=1}^M\partial_{x_k}(E_k\overline{\varphi})\Big)+o(\eps)\\
&\geq&A^{-1}\Big(\sigma(c-1)-3\eps\Big)\overline{\varphi}+o(\eps)>0,
\end{eqnarray*}
thus it holds
\begin{eqnarray*}
\label{yw1-t}
&&\int_{\RR^M}\Big[(\varphi-\overline{\varphi})(\del_t\del_{x_j}^sh)^2
+\varphi A^{-1}(B-|H|)|\nabla\del_{x_j}^sh|^2\nonumber\\
&&\quad
+\Big(-\overline{\varphi}_t
+A^{-1}(C-A)\overline{\varphi}-\sum_{k=1}^M\partial_{x_k}(A^{-1}E_k\overline{\varphi})\Big)(\del_{x_j}^sh)^2
\Big]dx\nonumber\\
&&\gtrsim\int_{\RR^M}\Big[c\overline{\varphi}(\del_t\del_{x_j}^sh)^2
+\varphi(\sigma-{1\over2})|\nabla\del_{x_j}^sh|^2
+\Big(\overline{\varphi}\Big(\sigma(c-1)-3\eps\Big)+o(\eps)\Big)(\del_{x_j}^sh)^2
\Big]dx.~~~~~~
\end{eqnarray*}

Similarly, using (\ref{NYN0-2})-(\ref{NYN0-6}), there exists a positive constant $c_0$ such that
\begin{eqnarray*}
\widetilde{A}_1(t,x)&:=&
A^{-1}\Big(\pdt^2(A\overline{\varphi})-\Delta(B\overline{\varphi})
-\pdt(C\overline{\varphi})
+\sum_{k,i=1}^M\partial_{x_i}\partial_{x_k}(H_{ki}\overline{\varphi})
+\sum_{k=1}^M\del_t\partial_{x_k}(E_k\overline{\varphi})\nonumber\\
&&\quad-\sum_{k=1}^M\partial_{x_k}(D_k\overline{\varphi})-\overline{\varphi}\Big)-A^{-1}\del_{t}^2A\overline{\varphi}-A^{-1}A_t\overline{\varphi}_t
-\sum_{k=1}^M\del_{x_k}^2( A^{-1})B\overline{\varphi}\nonumber\\
&&\quad-\sum_{k=1}^M\del_{x_k} A^{-1}\del_{x_k}( B\overline{\varphi})
-A^{-1}_tC\overline{\varphi}
+\sum_{k,i=1}^M\partial_{x_i}\partial_{x_k}A^{-1}H_{ki}\overline{\varphi}
+\sum_{k,i=1}^M\partial_{x_i}A^{-1}\partial_{x_k}(H_{ki}\overline{\varphi})\\
&&\quad
+\sum_{k,i=1}^M\partial_{x_k}A^{-1}\partial_{x_i}(H_{ki}\overline{\varphi})
+\sum_{k=1}^M\del_t\partial_{x_k}A^{-1}E_k\overline{\varphi}
+\sum_{k=1}^M\partial_{x_k}A^{-1}\del_t(E_k\overline{\varphi})\\
&&\quad
+\sum_{k=1}^M\del_tA^{-1}\partial_{x_k}(E_k\overline{\varphi})
-\sum_{k=1}^M\partial_{x_k}A^{-1}D_k\overline{\varphi}
+A^{-1}\overline{\varphi}\\
&&\quad-2\overline{\varphi}\Big({1\over2}
+\sum\limits_{\substack{s_1+s_2=s\\1\leq s_1\leq s\\0\leq s_2\leq s-1}}|\del_{x_j}^{s_1}(A^{-1}B)|
+\sum\limits_{\substack{s_1+s_2=s\\1\leq s_1\leq s\\0\leq s_2\leq s-1}}|\del_{x_j}^{s_1}(A^{-1}C)|\\
&&\quad+\sum_{k=1}^M\sum\limits_{\substack{s_1+s_2=s\\1\leq s_1\leq s\\0\leq s_2\leq s-1}}|\del_{x_j}^{s_1}(A^{-1}D_k)|
+\sum_{k=1}^M\sum\limits_{\substack{s_1+s_2=s\\1\leq s_1\leq s\\0\leq s_2\leq s-1}}|\del_{x_j}^{s_1}(A^{-1}E_k)|\\
&&\quad+\sum_{k,i=1}^M\sum\limits_{\substack{s_1+s_2=s\\1\leq s_1\leq s\\0\leq s_2\leq s-1}}|\del_{x_j}^{s_1}(A^{-1}H_{ki})|\Big)\\
&\geq&A^{-1}\Big(A\overline{\varphi}_{tt}-B\Delta\overline{\varphi}-C\overline{\varphi}_t-c_0\eps\overline{\varphi}\Big)+o(\eps)\nonumber\\
&\geq&\Big(c^2\sigma-c_0(\sigma+1)\eps\Big)\overline{\varphi}+o(\eps)>0,\\
\end{eqnarray*}
and
\begin{eqnarray*}
&&\widetilde{A}_2(t,x)\\
&:=&
A^{-1}\Big(-\del_t(A\varphi)
-\sum_{k=1}^M\partial_{x_k}(E_k\varphi)+(2C-|D|-1)\varphi-(2A+|E|)\overline{\varphi}
-\sum_{k=1}^M|\del_{x_k}(B\varphi)|
\Big)\nonumber\\
&&+A^{-1}A_t\varphi
-\sum_{k=1}^M\partial_{x_k}A^{-1}E_k\varphi
+A^{-1}\varphi-\sum_{k=1}^M|\del_{x_k}A^{-1}B\varphi|
-2\varphi\Big({1\over2}+\sum\limits_{\substack{s_1+s_2=s\\1\leq s_1\leq s\\0\leq s_2\leq s-1}}|\del_{x_j}^{s_1}(A^{-1}B)|\\
&&+\sum\limits_{\substack{s_1+s_2=s\\1\leq s_1\leq s\\0\leq s_2\leq s-1}}|\del_{x_j}^{s_1}(A^{-1}C)|
+\sum_{k=1}^M\sum\limits_{\substack{s_1+s_2=s\\1\leq s_1\leq s\\0\leq s_2\leq s-1}}|\del_{x_j}^{s_1}(A^{-1}D_k)|\nonumber\\
&&+\sum_{k=1}^M\sum\limits_{\substack{s_1+s_2=s\\1\leq s_1\leq s\\0\leq s_2\leq s-1}}|\del_{x_j}^{s_1}(A^{-1}E_k)|
+\sum_{k,i=1}^M\sum\limits_{\substack{s_1+s_2=s\\1\leq s_1\leq s\\0\leq s_2\leq s-1}}|\del_{x_j}^{s_1}(A^{-1}H_{ki})|\Big),\\
\\
&\geq&c(c\sigma-1-5\eps)\overline{\varphi}+o(\eps)>0,\quad for~~ c\sigma>1+5\eps,
\end{eqnarray*}
and
\begin{eqnarray*}
\widetilde{A}_3(t,x)&:=&
A^{-1}\Big((2B-2|H|-|E|)\overline{\varphi}-\del_t(B\varphi)-\varphi|D|-\sum_{k=1}^M|\del_{x_k}(B\varphi)|\Big)\\
&&
-A^{-1}_tB\varphi-\sum_{k=1}^M|\del_{x_k}A^{-1}B\varphi|\\
&\geq&(2\sigma+c^2\sigma-1-c_0\eps)\overline{\varphi}+o(\eps)>0,
\end{eqnarray*}
thus it holds
\begin{eqnarray*}\label{yw1-A1}
\frac{1}{2}\int_{\RR^M}\widetilde{A}_1(t,x)(\del_{x_j}^sh)^2dx
\geq\frac{1}{2}\int_{\RR^M}\Big[\Big(c^2\sigma-c_0(\sigma+1)\eps\Big)\overline{\varphi}+o(\eps)\Big](\del_{x_j}^sh)^2dx,\\
\frac{1}{2}\int_{\RR^M}\widetilde{A}_2(t,x)(\del_t\del_{x_j}^sh)^2dx
\geq\frac{1}{2}\int_{\RR^M}\Big[c(c\sigma-1-5\eps)\overline{\varphi}+o(\eps)\Big](\del_t\del_{x_j}^sh)^2dx,\\
\frac{1}{2}\int_{\RR^M}\widetilde{A}_3(t,x)|\nabla\del_{x_j}^s h|^2dx
\geq\frac{1}{2}\int_{\RR^M}\Big[\Big(2\sigma+c^2\sigma-1-c_0\eps\Big)\overline{\varphi}+o(\eps)\Big]|\nabla\del_{x_j}^s h|^2dx.
\end{eqnarray*}

Hence, upon above estimates of coefficients, we deduce \eqref{TT001sc1} into
\begin{eqnarray}
\label{yw1-AS}
{d\over dt}\sum_{j=1}^M\int_{\RR^M}\overline{\varphi}\Big((\del_t\del_{x_j}^sh)^2&+&|\nabla\del_{x_j}^s h|^2+(\del_{x_j}^sh)^2
\Big)dx\nonumber\\
&&\quad+C_{c,\sigma,\eps}\int_{\RR^M}\overline{\varphi}\Big((\del_t\del_{x_j}^sh)^2
+|\nabla\del_{x_j}^s h|^2+ (\del_{x_j}^sh)^2\Big)dx\nonumber\\
&\lesssim&\frac{1}{2}\sum_{j=1}^M\int_{\RR^M}\Big(\varphi+\overline{\varphi}\Big)(\del_{x_j}^s(A^{-1}f))^2dx+\sum_{j=1}^MR_j.
\end{eqnarray}

We notice that the term $R_j$ can be controlled by (\ref{wyw01-1}),
thus we can apply Gronwall's inequality to (\ref{yw1-AS}) to obtain
\begin{eqnarray*}
&&\sum_{j=1}^M\int_{\RR^M}\overline{\varphi}\Big((\del_t\del_{x_j}^sh)^2
+|\nabla\del_{x_j}^s h|^2+ (\del_{x_j}^sh)^2\Big)dx\\
&\lesssim& e^{-C_{c,\sigma,\eps}t}\sum_{j=1}^M\sum_{\theta=0}^s\Big[\int_{\RR^M}\Big((\del_{x_j}^{\theta}h(0,x))^2
+|\nabla\del_{x_j}^{\theta} h(0,x)|^2+ (\del_{x_j}^{\theta}h(0,x))^2\Big)dx\\
&&\quad+\int_0^{\infty}\int_{\RR^M}\Big(\varphi+\overline{\varphi}\Big)(\del_{x_j}^{\theta}f)^2dxdt\Big].
\end{eqnarray*}

\end{proof}

Directly deriven from lemma 2.3, we have the following result.
\begin{lemma}
Let $f\in \CC^1((0,\infty); H^s(\RR^M))$. Then the solution of linear wave equation (\ref{NYN0-1}) satisfies
\begin{eqnarray*}
\label{wyw01-0}
&&\sum_{j=1}^M\int_{\RR^M}\Big((\del_t\del_{x_j}^sh)^2
+|\nabla\del_{x_j}^s h|^2+(\del_{x_j}^sh)^2\Big)dx\nonumber\\
&\lesssim& e^{-\eps t}\sum_{j=1}^M\sum_{\theta=0}^{s}\Big[\int_{\RR^M}\Big((\del_{x_j}^{\theta}\del_th(0,x))^2
+|\nabla \del_{x_j}^{\theta}h(0,x)|^2+(\del_{x_j}^{\theta}h(0,x))^2\Big)dx\nonumber\\
&&+\int_0^{\infty}\int_{\RR^M}(\del_{x_j}^{\theta}f)^2dxdt\Big].
\end{eqnarray*}
\end{lemma}

Based on above results, the global existence of Sobolev regularity solution for the damped wave equation with variable coefficients (\ref{NYN0-1}) can be given.

\begin{proposition}
Let $f\in \CC^1((0,\infty); H^s(\RR^M))$. Assume that (\ref{NYN0-2})-(\ref{NYN0-4}) hold. The linear problem (\ref{NYN0-1}) admits a unique global solution
$$
h(t,x)\in\CC^1((0,\infty);H^{s}(\RR^M))\cap\CC((0,\infty);H^{s+1}(\RR^M)).
$$
Moreover, it satisfies
\bel{www0-1}
\aligned
\|h\|^2_{H^s}
&\lesssim e^{-\eps\tau}\Big(\|h_0\|^2_{H^s}+\|h_1\|^2_{H^{s-1}}+\|f\|_{H^s}^2\Big).
\endaligned
\ee
\end{proposition}
\begin{proof}
The proof iis based on the standard fixed point iteration by following the proof process of theorem 3.2 given in page 18 of \cite{Sog}. We sketch the proof.
Let $\phi=(h,w)^T$ and $w=h_{t}$. Then the linear equation (\ref{NYN0-1}) is equivalent to
\begin{equation}\label{0E3-32}
\frac{d}{dt}\phi(t)+\mathcal{A}\phi(t)=G(t,x),~~t>0,
\end{equation}
with the initial data
$$
\phi_0:=\phi(0,x)=(h_0(x),h_1(x))^T,
$$
 the operators $\mathcal{A}$ is independent of $t$, it takes the form
$$
\aligned
\mathcal{A}=\left(
\begin{array}{cccc}
0&1\\
\overline{\Delta}_1&\overline{\Delta}_2
\end{array}
\right),
\endaligned
$$
where
$$
\aligned
\overline{\Delta}_1&:=-A^{-1}B\Delta+A^{-1}\sum_{k=1}^MD_k(t,x)\partial_{x_k}+A^{-1}\sum_{k=1}^M\sum_{i=1}^MH_{ki}(t,x)\partial_{x_k}\partial_{x_i},\\
\overline{\Delta}_2&:=A^{-1}C+\sum_{k=1}^ME_k(t,x)\partial_{x_k},
\endaligned
$$
and
$$
\aligned
G(t,x)=\left(
\begin{array}{cccc}
0\\
A^{-1}f
\end{array}
\right).
\endaligned
$$

We now consider the approximation problem
$$
\phi(t,x)=\phi_0-\int_0^{t}\Big(\mathcal{A}\phi(s,x)+G(s,x)\Big)ds
$$
has a Cauchy sequence $\{\phi_{j}\}_{j\in\mathbb{Z}^+}$ in $H^s(\RR^M)$, whose limit is $\phi(t,x)$ and it solves the linearized system (\ref{0E3-32}) in $(0,T]$, where $T$ denotes a positive constant.
Furthermore, by means of the results of Lemma 2.1-2.4, it holds
$$
\aligned
\|h\|^2_{H^s}
&\lesssim e^{-\eps\tau}\Big(\|h_0\|^2_{H^s}+\|h_0\|^2_{H^{s-1}}+\|f\|_{H^s}^2\Big),
\endaligned
$$
thus the constructed local solution $\phi(t,x)$ can be extended to the global solution in time.  

To see the uniqueness, let $\phi_1$ and $\phi_2$ are two solutions of (\ref{0E3-32}) with the same data, then $\phi:=\phi_1-\phi_2$ admits zero Cauchy data, and
$$
\frac{d}{d\tau}\phi(t)+\mathcal{A}\phi(t)=0,\quad with~~\phi=(h,h_t)^T.
$$
Therefore, we can apply (\ref{www0-1}) to derive $\phi\equiv0$.
This completes the proof.

\end{proof}

\section{The linearized problem}\setcounter{equation}{0}

The timelike minimal surface equation (\ref{E1-1}) is equivalent to
\bel{am0-0}
\Fcal(u)=0,
\ee
where $\Fcal(u)$ is given in (\ref{n1-0}).
Let $w(t,x)$ be the perturbation,  we set
$$
u(t,x)=u_s(x)+w(t,x),
$$
where $u_s(x)$ is given in (\ref{E01-2R1}).
Inserting it into (\ref{am0-0}), it leads to the nonlinear perturbation equation
\begin{eqnarray}
\label{G3-1}
\Lcal w&:=&\Big(1+|A+\nabla w|^2\Big)w_{tt}-\Big(1+|A+\nabla w|^2-w_t^2\Big)\triangle w-\frac{1}{2}w_t\partial_t\Big((\nabla w+2A)\cdot\nabla w\Big)\nonumber\\
&&\quad+\frac{1}{2}\sum_{k=1}^M\Big(\partial_{x_k}w+A_k\Big)\partial_{x_k}\Big((\nabla w+2A)\cdot\nabla w-w_t^2\Big)=0,
\end{eqnarray}
with the small initial data
\bel{G3-a1}
\aligned
w(0,x)&:=w_0(x)=u_0(x)-u_s(x),\\
w_t(0,x)&:=w_1(x)=u_1(x),
\endaligned
\ee
and the vanishing boundary condition
\bel{G3-a2}
\lim_{|x|\rightarrow+\infty}w(t,x)=0.
\ee

Let constant $p>{1\over2}$ and the parameter $\eps$ be a small positive constant.
We linearize nonlinear equation (\ref{E0001-1}) at the initial approximation function $w^{(0)}(t,x)$, where we assume
\begin{eqnarray}\label{P0-1}
&&w^{(0)}(t,x):=t^2(t^4+1)^{-p}w^{(0)}(x),\quad with\quad A\cdot \nabla w^{(0)}(x)>0,\quad A\neq\textbf{0},\quad \forall x\in\RR^M,\\
\label{P0-2}
&&w^{(0)}(0,x)=\del_tw^{(0)}(t,x)|_{t=0}=0,\quad and \quad w^{(0)}(x)\in C^{\infty}(\RR^M),\\
\label{P0-3}
&&\|\del^sw^{(0)}\|_{\LL^{\infty}}\lesssim\eps,~~\forall s\in\{0,1,2,3,\ldots\}
\end{eqnarray}
 to get a nonhomogeneous linear damped wave equation with variable coefficients
\bel{E0002-1}
\aligned
\Lcal^{(0)}_{w^{(0)}}h^{(1)}:=A_1(t,x)h^{(1)}_{tt}&-A_2(t,x)\triangle h^{(1)}+\sum_{k=1}^M\sum_{i=1}^MA_{3ki}(t,x)\del_{x_i}\del_{x_k}h^{(1)}+A_4(t,x)h^{(1)}_t\\
&\quad+A_5(t,x)\cdot\nabla h^{(1)}+A_6(t,x)\cdot\nabla h^{(1)}_t=E^{(0)}(t,x),
\endaligned
\ee
where coefficients of it are
\begin{eqnarray*}
A_1(t,x)&:=&1+|A+t^2(t^4+1)^{-p}\nabla w^{(0)}|^2,\\
A_2(t,x)&:=&1+|A+t^2(t^4+1)^{-p}\nabla w^{(0)}|^2-4t^2(1+(1-2p)t^4)^2(t^4+1)^{-2(p+1)}(w^{(0)})^2,\\
A_{3ki}(t,x)&:=&\Big(t^2(t^4+1)^{-p}\del_{x_k}w^{(0)}+A_k\Big)\Big(t^2(t^4+1)^{-p}\nabla w^{(0)}+A_i\Big),\\
A_4(t,x)&:=&4t(t^4+1)^{-p-1}\Big((2p-1)t^4-1\Big)\Big[\Big(A+t^2(t^4+1)^{-p}\nabla w^{(0)}\Big)\cdot\nabla w^{(0)}\\
&&\quad
-t^2(t^4+1)^{-p}w^{(0)}\triangle w^{(0)}\Big],\\
A_5(t,x)&:=&t^2(t^4+1)^{-p}\Big[{1\over2}\nabla\Big(t^2(t^4+1)^{-p}|\nabla w^{(0)}|^2-4(t^4+1)^{-p-2}(1+(1-2p)t^4)^2(w^{(0)})^2\\
&&\quad+2A\cdot\nabla w^{(0)}\Big)
+\sum_{k=1}^M\Big(t^2(t^4+1)^{-p}\del_{x_k}w^{(0)}+A_k\Big)\nabla\del_{x_k}w^{(0)}\\
&&\quad+2\Big(t^2(t^4+1)^{-p}\nabla w^{(0)}+A\Big)
\Big(t^{-2}(t^4+1)^p(t^2(t^4+1)^{-p})''w^{(0)}-\triangle w^{(0)}\Big)\Big],\\
A_6(t,x)&:=&4t((2p-1)t^4-1)(t^4+1)^{-p-1}w^{(0)}\Big(t^2(t^4+1)^{-p}\nabla w^{(0)}+A\Big).
\end{eqnarray*}

By the assumption of (\ref{P0-3}), it holds
\begin{eqnarray}\label{i0-1}
&&\|A_1\|_{\LL^{\infty}}\sim 1+|A|^2+O(\eps),\quad \|A_2\|_{\LL^{\infty}}\sim 1+|A|^2+O(\eps),\quad\|A_3\|_{\LL^{\infty}}\sim |A|^2+O(\eps),~~~\\
\label{i0-2}
&& \|A_4\|_{\LL^{\infty}}\sim O(\eps)\quad \|A_5\|_{\LL^{\infty}}\sim O(\eps),\quad \|A_6\|_{\LL^{\infty}}\sim O(\eps),
\end{eqnarray}
and the leading term in $A_4(t,x)$ is
\bel{i0-3}
2t(t^4+1)^{-p-1}\Big((2p-1)t^4-1\Big)A\cdot\nabla w^{(0)}>0,\quad for ~~t>T_p:=(2p-1)^{-{1\over4}},
\ee
from it, one can see the coefficient $A_4(t,x)>0$, and the linear equation $(\ref{E0002-1})$ is dissipative for the time $t>T_p:=(2p-1)^{-{1\over4}}$.

\subsection{The first aproximation step}

We introduce a family of smooth operators in the smooth bounded domain. We refer to \cite{Alin,AP} for more details.
{Let $\chi\in\CC_0^{\infty}(\RR^M)$ such that $\chi=1$ in $\Big\{x\in\RR^+ \Big| |x|\leq{1\over2}\Big\}$, otherwise, $\chi\equiv0$. We follow Proposition 1.6 of \cite[p. 83]{AP} or page 72 of \cite{K} to define
the smooth operator $\Pi^{(1)}_{\theta_m}$  by
$$
\Pi^{(1)}_{\theta_m}U=\sum_{m}\chi({|x|\over\theta_m})U(t,x),
$$
and the smooth operator $\Pi^{(2)}_{\theta_m'}$  by
$$
\Pi^{(2)}_{\theta_m'}U=\sum_m\chi({t\over\theta_m'})U(t,x),
$$
we define
$$
\Pi_{\theta_m,\theta_m'}:=\Pi^{(1)}_{\theta_m}\Pi^{(2)}_{\theta_m'},
$$
then one can verify
\begin{eqnarray}\label{AH1}
&&\|\Pi_{\theta_m,\theta_m'}\del_t^jU\|_{H^{s_1}}\leq C\theta_m^{(s_1-s_2)_+}(\theta'_m)^{(j-j')_+}\|\del_t^{j'}U\|_{H^{s_2}},~~\forall s_1,~s_2\geq0,\\
&&\|\Pi_{\theta_m,\theta_m'}\del_t^jU-\del_t^jU\|_{H^{s_1}}\leq C\theta_m^{s_1-s_2}(\theta_m')^{j-j'}\|\del_t^{j'}U\|_{H^{s_2}},~~0\leq s_1\leq s_2,\nonumber
\end{eqnarray}
where $C$ is a positive constant and $(s_1-s_2)_+:=\max(0,s_1-s_2)$. The proof follows from
the proof given in \cite[p. 192]{Alin} or \cite{AP}.

In our iteration scheme, we set
$$
\theta_m=\theta_m'=N_m=N_0^m,\quad \forall m= 0,1,2,\ldots,
$$
where $N_0$ is a fixed positive constant, we denote it by $\Pi_{N_m}$ for convenience,
then by (\ref{AH1}), it holds
\begin{align}\label{AH2}
\|\Pi_{N_m}\del_t^{s_1}U\|_{H^{s_1}}\lesssim N_m^{2(s_1-s_2)}\|\del_{t}^{s_2}U\|_{H^{s_2}},\quad \forall s_1\geq s_2.
\end{align}

Let us consider the linear damped wave equation with variable coefficient derive by an external force as follows
\bel{Y1-1}
\aligned
&\Pi_{N_1}\Lcal^{(0)}_{w^{(0)}}h^{(1)}=\Pi_{N_1}E^{(0)}(t,x),\quad \forall (t,x)\in\RR^+\times\RR^M,\\
& h^{(1)}(0,x)=h^{(1)}_0,\quad h^{(1)}_{t}(0,x)=h^{(1)}_1,
\endaligned
\ee
where the external force $E^{(0)}(t,x)$ is related to the error term at the initial approximation function.

Furthermore, let $\sigma\in(1,1+|A|^2)$, then by the form of $A_i(t,x)$ with $i=1,2,\ldots,6$, it follows from  (\ref{i0-1})- (\ref{i0-3}) that
\bel{00NYN0-2}
\sigma_0>A_1(t,x)>\sigma>1,\quad A_2(t,x)>\sigma>1,\quad A_4(t,x)>0, \quad\forall (x,t)\in\RR^M\times\RR^+,
\ee
and
\bel{00NYN0-3}
|A_3|:=\sum_{k,i=1}^M|A_{3ki}|\leq\sigma\quad and \quad A_{3ki}>0,
\ee
and
\bel{00NYN0-4}
|A_5|:=\sum_{k=1}^M|A_{5k}|\sim\eps,\quad |A_6|:=\sum_{k=1}^M|A_{6k}|\sim\eps,
\ee
and
\bel{00NYN0-04}
\aligned
&\|\del^sA_1\|_{\LL^{\infty}}\sim\eps,\quad \|\del^sA_2\|_{\LL^{\infty}}\sim\eps,\quad \|\del^sA_4\|_{\LL^{\infty}}\sim\eps,\quad \|\del^sA_{3ki}\|_{\LL^{\infty}}\sim\eps,\\
&\|\del^sA_{5k}\|_{\LL^{\infty}}\sim\eps,\quad \|\del^sA_{6k}\|_{\LL^{\infty}}\sim\eps,\quad\forall s\in\{1,2,3,\ldots\},
\endaligned
\ee
which means that assumptions (\ref{NYN0-2})-(\ref{NYN0-4}) hold. Thus we can use proposition 2.1 to
obtain the global exsitence of solution for the linear equation (\ref{Y1-1}).
\begin{proposition}
Let $E^{(0)}\in \CC^1((0,\infty); H^s(\RR^M))$. The linear problem (\ref{NYN0-1}) admits a unique global solution
$$
h^{(1)}(t,x)\in\CC^1((0,+\infty);H^{s}(\Omega)).
$$
Moreover, it satisfies
\bel{0Aa1-0}
\aligned
\|h^{(1)}\|^2_{H^s}
&\lesssim e^{-\eps t}\Big(\|h^{(1)}_0\|^2_{H^s}+\|h^{(1)}_1\|^2_{H^{s-1}}+\|E^{(0)}\|_{H^s}^2\Big).
\endaligned
\ee
\end{proposition}

\subsection{The general approximation step}

Let constant $0<\eps\ll 1$, we define

$$
\Bcal_{\eps}:=\{ w^{(i)}\in H^{s}(\RR^M):\quad \|w^{(i)}\|_{H^s}\leq \eps \},
$$
with the integer $2\leq i\leq m-1$.

Assume that the $m$-th approximation step of (\ref{G3-1}) is denoted by
$h^{(m)}(t,x)$ with $m=2,3,\ldots$, where we set
$$
h^{(m)}(t,x):=w^{(m)}(t,x)-w^{(m-1)}(t,x),
$$
then it holds
$$
w^{(m)}(t,x)=w^{(0)}(t,x)+h^{(1)}(t,x)+\sum_{i=2}^mh^{(i)}(t,x).
$$

We linearize the nonlinear equations (\ref{G3-1}) around $w^{(m-1)}(t,x)$ to get the following initial value problem
\bel{AYY1-1}
\left\{
\begin{array}{lll}
&\Pi_{N_m}\Lcal^{(m)}_{w^{(m-1)}}h^{(m)}=\Pi_{N_m}E^{(m-1)}(\tau,x),\quad \forall (t,x)\in\RR^+\times\RR^M,\\
& h^{(m)}(0,x)=h^{(m)}_0,\quad h^{(m)}_{\tau}(0,x)=h^{(m)}_1,
\end{array}
\right.
\ee
where the error term is
\bel{E4-10R2}
E^{(m-1)}:=\Lcal[w^{m-1}(t,x)]=\Rcal(h^{(m)}),
\ee
and
\bel{AYY1-1-0}
\aligned
\Rcal(h^{(m)})&:=\Lcal(w^{(m-1)}+h^{(m)})-\Lcal(u^{(m-1)})
-\Lcal^{(m)}_{w^{(m-1)}}h^{(m)},
\endaligned
\ee
which is also the nonlinear term in the approximation problem (\ref{G3-1}) at $w^{(m-1)}(t,x)$.

Similar to the process of getting proposition 3.1, we can
construct the $m$-th approximation solution.

\begin{proposition}
Let $E^{(m-1)}\in \CC^1((0,\infty); H^s(\RR^M))$.
Assume $w^{(m-1)}\in \Bcal_{\eps}$.  Then
the linearized problem (\ref{AYY1-1})
admits a unique global solution
$$
h^{(m)}(t,x)\in\CC^1((0,+\infty);H^s(\RR^M)),
$$
which satisfies
\bel{E4-10R1}
\|h^{(m)}\|^2_{H^s}
\lesssim e^{-\eps t}\Big(\|h^{(m)}_0\|^2_{H^s}+\|h^{(m)}_1\|^2_{H^{s-1}}+\|E^{(m-1)}\|_{H^s}^2\Big).
\ee
\end{proposition}
\begin{proof}
On one hand, we find the $m$-th $(m\geq2)$ approximation solution $w^{(m)}(t,x)$, which is equivalent to find $h^{(m)}(t,x)$ such that
\bel{A4-7}
w^{(m)}(t,x)=w^{(m-1)}(t,x)+h^{(m)}(t,x).
\ee
Substituting (\ref{A4-7}) into (\ref{G3-1}), it holds
$$
\Lcal(w^{(m)})=\Lcal(w^{(m-1)})+\Lcal^{(m)}_{w^{(m-1)}}h^{(m)}+\Rcal(h^{(m)}),
$$
then let
$$
\Lcal^{(m)}_{w^{(m-1)}}h^{(m)}=-\Lcal(w^{(m-1)})=-E^{(m-1)},
$$
which is a linear damped wave equation taking the form of (\ref{E0002-1}) by replacing $w^{(0)}$ with $w^{(m-1)}$, and the error term
$$
E^{(m-1)}:=\Lcal(w^{(m-1)})=\Rcal(h^{(m)}).
$$

On the other hand, by direct computation we find
\bel{AAE2-6}
\del^sw^{(m-1)}(t,x)=\del^s w^{(0)}(t,x)+\del^{s} h^{(1)}(t,x)+\sum_{i=2}^{m-1}\del^{s} h^{(i)}(t,x),\quad \forall s\in\NN,
\ee
where we use the symbol $\del^s$ to denote the $s$-th derivatives of time or spacial variables,
and for a sufficient small positive parameter $\eps$, it holds
$$
\|h^{(i)}\|_{H^s}\lesssim\eps,
$$
thus we can see that
$$
\del^sw^{(m-1)}(t,x)\sim \del^sw^{(0)}(t,x)+O(\eps),
$$
it means that the leading term of the $m$th approximation solution is the initial approximation function $w^{(0)}(t,x)$.
Thus there is the same structure between the linear system (\ref{Y1-1}) and the linear system of $m$th approximation solutions, and a similar assumption given in (\ref{00NYN0-2})-(\ref{00NYN0-4}) can be satisfied.
By means of the same arguments as in the proof of Proposition 3.1, we can show that the linear problem (\ref{AYY1-1}) admits a global solution $h^{(m)}(t,x)\in\CC^1((0,+\infty);H^s(\RR^M))$.
Meanwhile, (\ref{E4-10R1}) can be obtained. The proof is now complete.
\end{proof}


\section{The nonlinear problem}\setcounter{equation}{0}

In this section, our target is to prove that $w^{(\infty)}(t,x)$ is a global solution of the nonlinear equations (\ref{G3-1}). This is equivalent to show that the series
$\sum\limits_{i=1}^mh^{(i)}(t,x)$ is convergent. We now give the tame estimate of error term in each iteration scheme.
\begin{lemma}
 The error term verifies
\bel{E4-9R1}
\|\Pi_{N_m}E^{(m-1)}\|_{H^s}
=\|\Pi_{N_m}\Rcal(h^{(m)})\|_{H^s}
\lesssim N_m^{4s}\|h^{(m)}\|^2_{H^s},\quad for~s\geq1.
\ee
\end{lemma}
\begin{proof}
We notice that the error term is
\begin{eqnarray}\label{nnn0-1}
\Rcal(h^{(m)})&:=&\Big(2A\cdot\nabla h^{(m)}+|\nabla h^{(m)}|^2\Big)h^{(m)}_{tt}-\Big(2A\cdot\nabla h^{(m)}+|\nabla h^{(m)}|^2-(h^{(m)}_t)^2\Big)\triangle h^{(m)}\nonumber\\
&&-\frac{1}{2}h^{(m)}_t\partial_t\Big((\nabla h^{(m)}+2A)\cdot\nabla h^{(m)}\Big)+\frac{1}{2}\sum_{k=1}^M\Big(\partial_{x_k}h^{(m)}+A_k\Big)\partial_{x_k}\Big(|\nabla h^{(m)}|^2-(h^{(m)}_t)^2\Big)\nonumber\\
&&\quad+\sum_{k=1}^M\partial_{x_k}h^{(m)}\Big(A\cdot\nabla \partial_{x_k}h^{(m)}\Big),
\end{eqnarray}
and the highest order of derivatives on $x$ and $t$ of it
is $4$ and $2$, respectively. Since the solution of (\ref{AYY1-1}) should be constructed in $\Bcal_{\eps}$, it holds
$$
\|h^{(m)}\|^p_{H^s}\leq \|h^{(m)}\|^2_{H^s},\quad for\quad p\geq2.
$$
Thus we apply Cauchy's inequality and (\ref{nnn0-1}) to estimate each term in $\Rcal(h^{(m)})$, it holds
$$
\|\Pi_{N_m}\Rcal(h^{(m)})\|_{H^s}\lesssim N_m^{4s}\|h^{(m)}\|^2_{H^s},\quad for~s\geq1.
$$

\end{proof}

We now show the convergence of iteration scheme.
For any $s>1$, let $1\leq\bar{k}< k_0\leq k\leq s$ and
$$
\aligned
&k_m:=\bar{k}+\frac{k-\bar{k}}{2^m},\\
&\alpha_{m+1}:=k_m-k_{m+1}=\frac{k-\bar{k}}{2^{m+1}},
\endaligned
$$
which gives that
\bel{EX1-1}
k_0>k_1>\ldots>k_m>k_{m+1}>\ldots.
\ee
\begin{proposition}
The nonlinear problem (\ref{G3-1}) with the initial data (\ref{G3-a1}) and the boundary condition
(\ref{G3-a2})
admits a unique global Sobolev regularity solution
$$
w^{(\infty)}(t,x)=w^{(0)}(t,x)+\sum_{m=1}^{\infty}h^{(m)}(t,x)+e^{-t^2}  w_0(x)+ te^{-t}w_1(x).
$$
\end{proposition}
\begin{proof}

The proof is based on the induction.
For convenience, we first deal with the case of zero initial data, that is, $w_0(x)=0$ and $w_1(x)=0$. After that, we discuss the case $w_0(x)\not\equiv0$ and $w_1(x)\not\equiv0$.
Note that $N_m=N_0^m$ with $N_0>1$. For all $m=1,2,\ldots$, we claim that there exists a sufficient small positive constant $\eps$ such that
\bel{E4-12}
\aligned
&\|\Pi_{N_m}h^{(m)}\|_{H^{k_{m-1}}}<\eps^{2^{m-1}},\\
&\|\Pi_{N_m}E^{(m)}\|_{H^{k_{m-1}}}<\eps^{2^{m}},\\
&w^{(m)}\in\Bcal_{\eps}.
\endaligned
\ee

For the case of $m=1$, by (\ref{E4-10R1}), letting $0<\eps_0<N_0^{-(64+k-\bar{k})}\eps^2\ll1$, it holds
$$
\aligned
\|\Pi_{N_1}h^{(1)}\|_{H^{k_0}}&\lesssim N_1^{k_0-1}\|h^{(1)}\|_{H^1}\\
&\lesssim  N_1^{k_0-1}\|\Pi_{N_1}E^{(0)}\|_{H^1}\\
&\lesssim N_1^{2(k_0-1)}\|\Pi_{N_1}E^{(0)}\|_{H^{k_0}}\\
&<\eps_0<\eps^2.
\endaligned
$$
Moreover, by (\ref{E4-9R1}) and the above estimate,
$$
\|\Pi_{N_1}E^{(1)}\|_{H^{k_0}}\lesssim
\|\Pi_{N_1}\Rcal(h^{(1)})\|_{H^1}\lesssim  N_1^{4k_0}\|h^{(1)}\|^2_{H^{1}}<\eps^2,
$$
and
$$
\|w^{(1)}\|_{H^1}\lesssim\|w^{(1)}\|_{H^{k_0}}\lesssim\|w^{(0)}\|_{H^{k_0}}+\|h^{(1)}\|_{H^{k_0}}\lesssim
\eps,
$$
which means that $w^{(1)}\in\Bcal_{\eps}$.

Assume that the case of $m-1$ holds, that is,
\bel{E4-13}
\aligned
&\|\Pi_{N_{m-1}}h^{(m-1)}\|_{H^{k_{m-2}}}<\eps^{2^{m-2}},\\
&\|\Pi_{N_{m-1}}E^{(m-1)}\|_{H^{k_{m-2}}}<\eps^{2^{m-1}},\\
&w^{(m-1)}\in\Bcal_{\eps}.
\endaligned
\ee
Then we prove that the case of $m$ holds. Upon (\ref{E4-10R1}) and the second inequality of (\ref{E4-13}), we derive
\bel{E4-14R1}
\aligned
\|\Pi_{N_m}h^{(m)}\|_{H^{k_{m-1}}}&\lesssim N_m^{k_{m-1}-1}\|h^{(m)}\|_{H^1}\\
&\lesssim N_m^{k_{m-1}-1}\|\Pi_{N_m}E^{(m-1)}\|_{H^{1}}\\
&\lesssim N_m^{2(k_{m-1}-1)}\|\Pi_{N_m}E^{(m-1)}\|_{H^{k_{m-1}}}\\
&<\eps^{2^{m-2}},
\endaligned
\ee
which combined with (\ref{E4-9R1})-(\ref{EX1-1}) yields
\bel{E4-14}
\aligned
\|\Pi_{N_m}E^{(m)}\|_{H^{k_{m}}}
&\lesssim  N_m^{4}\|h^{(m)}\|^2_{H^{1}}\\
&\lesssim  N_{m}^{4}\Big(\|\Pi_{N_m}E^{(m-1)}\|_{H^{1}}\Big)^2\\
&\lesssim  N_{m}^{4}\Big(\|\Pi_{N_{m-1}}E^{(m-1)}\|_{H^{k_{m-1}}}\Big)^2\\
&\lesssim N_0^{(4+\alpha_{m+1})m+2(4+\alpha_{m+2})(m-1)}\Big(\|\Pi_{N_{m-2}}E^{(m-2)}\|_{H^{k_{m+2}}}\Big)^{2^2}\\
&\lesssim \ldots,\\
&\lesssim \Big( N_0^{64+k-\bar{k}}\|\Pi_{N_1}E^{(0)}\|_{H^{k_{2m}}}\Big)^{2^m}.
\endaligned
\ee

We choose a sufficiently small positive constant $\eps_0$ such that
$$
0< N_0^{64+k-\bar{k}}\|\Pi_{N_1}E^{(0)}\|_{H^{\bar{k}}}<\eps^2.
$$
Thus, by (\ref{E4-14}) we have
$$
\|\Pi_{N_m}E^{(m)}\|_{H^{k_{m}}}<\eps^{2^{m}},
$$
so, the error term goes to $0$ as $m\rightarrow\infty$, that is,
$$
\lim_{m\rightarrow+\infty}\|\Pi_{N_m}E^{(m)}\|_{H^{k_{m}}}=0.
$$

On the other hand, note that $N_m=N_0^m$, by (\ref{E4-13})-(\ref{E4-14R1}). It follows that
$$
\|w^{(m)}\|_{H^{k_{m}}}\lesssim \|w^{(m-1)}\|_{H^{k_{m}}}+\|h^{(m)}\|_{H^{k_{m}}}\lesssim\eps.
$$
This means that $w^{(m)}\in\Bcal_{\eps}$. Hence we conclude that (\ref{E4-12}) holds.

Therefore, the nonlinear equation (\ref{G3-1}) with the zero initial data and the boundary condition
(\ref{G3-a2})
admits a global Sobolev solution
$$
w^{(\infty)}(t,x)=w^{(0)}(t,x)+\sum_{m=1}^{\infty}h^{(m)}(t,x),
$$
and we use (\ref{P0-2}) to get
$$
w^{(0)}(0,x)=\del_tw^{(0)}(t,x)|_{t=0}=0.
$$

Next, we discuss the case of small non-zero initial data.

We introduce the auxiliary function
$$
\overline{w}(t,x)=w(t,x)-e^{-t^2}w_0(x)- te^{-t}w_1(x),\quad \forall x\in\RR^M.
$$
Thus, the initial data reduces to
$$
\overline{w}(0,x)=0,\quad \del_{t}\overline{w}(0,x)=0,
$$
and equations (\ref{E1-1}) are transformed into equations of $\overline{w}$.

Thus, we can follow the above iteration scheme to construct a global Sobolev solution
$\overline{w}$.
Furthermore, the global Sobolev solution of equations (\ref{E1-1}) with a small non-zero initial data takes the form
$\overline{w}(t,x)+ e^{-t^2} w_0(x)+ te^{-t}w_1(x)$, and this solution is uniqueness due to the uniqueness of each iteration step $h^{(m)}(t,x)$.
This completes the proof.
\end{proof}

\textbf{Acknowledgments.}
This work is supported by NSFC No 11771359, No 12161006, Guangxi Natural Science Foundation No 2021JJG110002 and the special foundation for Guangxi Ba Gui Scholars.\\



\end{document}